\documentclass[reqno,11pt,a4paper]{article}

\usepackage{amsmath,amssymb,amsthm}
\usepackage[UKenglish]{babel}
\usepackage{bbm}
\usepackage[right=2.54cm,left=2.54cm]{geometry}
\usepackage{ifthen}
\usepackage{multicol}
\usepackage{sectsty}
\usepackage[pdftitle={Feynman--Kac perturbation of C* quantum stochastic flows},
            pdfauthor={Alexander C. R. Belton and Stephen J. Wills},
            bookmarks=true,
            colorlinks=false,
            pdfstartview=FitH]{hyperref}
\usepackage{comment}

\sectionfont{\large}
\subsectionfont{\normalsize}

\mathchardef\ordinarycolon\mathcode`\:
\mathcode`\:=\string"8000
\def\vcentcolon{\mathrel{\mathop\ordinarycolon}}
\begingroup \catcode`\:=\active
\lowercase{\endgroup
\let :\vcentcolon}

\theoremstyle{plain}
\newtheorem{theorem}{Theorem}[section]
\newtheorem{lemma}[theorem]{Lemma}
\newtheorem{proposition}[theorem]{Proposition}

\theoremstyle{definition}
\newtheorem{definition}[theorem]{Definition}
\newtheorem{example}[theorem]{Example}
\newtheorem{notation}[theorem]{Notation}

\newtheorem{remark}[theorem]{Remark}

\newenvironment{mylist}%
{\begin{list}{}%
{\leftmargin 3.5em\labelwidth 3em\rightmargin 1em%
\topsep 0.5ex\itemsep 0.25ex}}%
{\end{list}}
{\begin{list}{}{}}{\end{list}}

\let\origthebibliography=\thebibliography
\def\thebibliography{\renewcommand{\section}[2]{}\origthebibliography}

\newcommand{\alg}{\mathsf{A}}
\newcommand{\blg}{\mathsf{B}}
\newcommand{\words}{\mathsf{W}}
\newcommand{\cuntz}{\mathcal{O}}
\newcommand{\algten}{\mathbin{\underline{\otimes}}}
\newcommand{\bop}[2]%
{\ifthenelse{\equal{#2}{}}{\bopp( #1 )}{\bopp( #1; #2 )}}
\newcommand{\bopp}{B}
\newcommand{\bra}[1]{\langle#1|}
\newcommand{\bt}{\mathbf{t}}
\newcommand{\bw}{\mathbf{w}}
\newcommand{\cb}{\text{cb}}
\newcommand{\comp}{\mathbin{\circ}}
\newcommand{\cP}{\mathcal{P}}
\newcommand{\cQ}{\mathcal{Q}}
\newcommand{\dom}{\mathop{\mathrm{dom}}}
\newcommand{\dyad}[2]{\ket{#1}\bra{#2}}
\newcommand{\elltwo}{L^2( \R_+; \mul )}
\newcommand{\evec}[1]{\evecc(#1)}
\newcommand{\evecc}{\varepsilon}
\newcommand{\evecs}{\mathcal{E}}

\newcommand{\finsubset}{\subset\subset}
\newcommand{\fock}{\mathcal{F}}
\newcommand{\hilb}{\mathsf{H}}
\newcommand{\hilc}{\mathsf{K}}
\newcommand{\hlf}{\mbox{$\frac12$}}
\newcommand{\id}{\mathrm{id}}
\newcommand{\im}{\mathop{\mathrm{im}}}
\newcommand{\ini}{\mathsf{h}}
\newcommand{\ket}[1]{|#1\rangle}
\newcommand{\matten}{\mathbin{\otimes_\mathrm{m}}}
\newcommand{\mmul}{{\wh{\mul}}}
\newcommand{\mul}{\mathsf{k}}
\newcommand{\nvec}[1]{\nvecc(#1)}
\newcommand{\nvecc}{\varpi}
\newcommand{\opsp}{\mathsf{V}}
\newcommand{\opsq}{\mathsf{W}}

\newcommand{\pw}{p_w}
\newcommand{\rd}{\mathrm{d}}
\newcommand{\std}{\,\rd}
\newcommand{\step}{L^{\mathrm{step}}( \R_+; \tot )}
\newcommand{\tfn}[1]{\mathbbm{1}_{#1}}
\newcommand{\tot}{\mathsf{T}}
\newcommand{\uwkten}{\mathbin{\overline{\otimes}}}

\newcommand{\wh}[1]{\widehat{#1}}
\newcommand{\wt}[1]{\widetilde{#1}}
\newcommand{\zp}{{}^z p}

\newcommand{\C}{\mathbb{C}}
\newcommand{\I}{\mathrm{i}}
\newcommand{\N}{\mathbb{N}}
\newcommand{\R}{\mathbb{R}}
\newcommand{\Q}{\mathbb{Q}}
\newcommand{\T}{\mathbb{T}}
\newcommand{\Z}{\mathbb{Z}}

\renewcommand{\ge}{\geqslant}
\renewcommand{\le}{\leqslant}

\newcommand{\tu}[1]{\textup{#1}}

\newcommand{\tinymaths}[1]{\mbox{\small $#1$}}

\DeclareMathOperator{\lin}{lin}

\numberwithin{equation}{section}

\begin{document}

\begin{center}
{\LARGE Feynman--Kac perturbation of $C^*$~quantum stochastic flows}
\begin{multicols}{2}
{\large Alexander C.~R.~Belton}\\[0.5ex]
{\small School of Engineering, Computing and Mathematics\\
University of Plymouth, United Kingdom\\[0.5ex]
\textsf{alexander.belton@plymouth.ac.uk}}
\columnbreak

{\large Stephen J.~Wills}\\[0.5ex]
{\small School of Mathematical Sciences\\
University College Cork, Ireland\\[0.5ex]
\textsf{s.wills@ucc.ie}}
\end{multicols}
{\small 18th April 2024}

\vspace{1ex}
\emph{Dedicated to the memory of K.~R.~Parthasarathy, master of
probability\\and a great inspiration to us both}
\end{center}

\begin{abstract}\noindent
The method of Feynman--Kac perturbation of quantum stochastic
processes has a long pedigree, with the theory usually developed
within the framework of processes on von~Neumann algebras. In this
work, the theory of operator spaces is exploited to enable a
broadening of the scope to flows on $C^*$~algebras. Although the
hypotheses that need to be verified in this general setting may seem
numerous, we provide auxiliary results that enable this to be
simplified in many of the cases which arise in practice. A wide
variety of examples is provided by way of illustration.
\end{abstract}

\section{Introduction}

Very early in the study of quantum stochastic processes on operator
algebras, it was realised in pioneering work of Accardi \cite{Acc78}
that analogues of the method of Feynman--Kac perturbation in
classical probability could be used to perturb Markov semigroups with
cocycles to give new semigroups, giving meaning to a
formal sum of generators.

The creation of quantum stochastic calculus in the early 1980s gave a
means of constructing such a cocycle $j$ on a $*$-algebra
$\alg \subseteq \bop{\ini}{}$ by solving the quantum stochastic
differential equation
\begin{equation}\label{eqn:EHeqn}
j_0( x ) = x \otimes I_\fock \qquad \text{and} \qquad %
\rd j_t( x ) = \wt{\jmath}_t\bigl( \phi( x ) \bigr) \std \Lambda_t %
\qquad ( x \in \alg_0 \subseteq \alg ),
\end{equation}
the cocycle in this case being used to perturb the ampliated CCR flow
$\sigma$ on Boson Fock space over $\elltwo$. Here $\ini$ and $\mul$
are known as the initial space and multiplicity space, respectively,
and $\alg_0$ is a dense, unital $*$-subalgebra of $\alg$. Solving this QSDE
is possible for any completely bounded generator $\phi$ \cite{LiW01}, and
this gives a cocycle in a generalised sense, but for many purposes it
is desirable that $j$ be $*$-homomorphic. Conversely, a completely
positive and contractive cocycle $j$ on a $C^*$ algebra $\alg$ that
satisfies the additional continuity requirement of being
\emph{Markov regular} or \emph{elementary} necessarily
satisfies~(\ref{eqn:EHeqn}), in which case the stochastic generator
$\phi$ must be completely bounded \cite{LiW00a, LiW00b}; see also
\cite[Theorem~6.4]{LiW21}.

Consequently it has become a folklore result that solutions
of the QSDE~(\ref{eqn:EHeqn}) and cocycles should be essentially the same
thing. However, as soon as one loosens the restrictive continuity or
boundedness assumptions, the correspondence between solutions
of~(\ref{eqn:EHeqn}) and cocycles is a lot less clear, especially
because of the many difficulties inherent in solving~(\ref{eqn:EHeqn})
for an unbounded generator~$\phi$. Our previous work \cite{BeW14} gave a method for
solving~(\ref{eqn:EHeqn}) on a $C^*$ algebra for certain unbounded stochastic
	generators, with the solution being $*$-homomorphic and a cocycle. This
involved two assumptions: a type of domain invariance for
$\phi$ and growth estimates for the iterates of $\phi$ whose
existence follows from the first assumption. In that paper we gave a
range of interesting examples where these assumptions held, but these
assumptions are obviously still somewhat restrictive.

In this paper we now follow the work of Evans and Hudson \cite{EvH90}
for perturbing a cocycle $j$, called the free flow, of the CCR flow
$\sigma$ to give a new cocycle $k$ for $\sigma$. This involves finding
bounded solutions to the multiplier equation, which is the following
operator-valued QSDE with time-dependent coefficients:
\begin{equation}\label{eqn:multeq}
X_0 = I_{\ini \otimes \fock} \qquad \text{and} \qquad %
\rd X_t = \wt{\jmath}_t( F ) \wt{X}_t \std \Lambda_t.
\end{equation}
The multiplier generator $F \in \alg \otimes \bop{\mmul}{}$, where
$\mmul := \C \oplus \mul$. The cocycle $k$ is then given by setting
$k_t( x ) := X^*_t j_t( x ) X_t$. The process obtained by
conjugation with $X$ is, at least formally, a cocycle of the
semigroup $J = (\wh{\jmath}_t \circ \sigma_t)_{t \in \R_+}$; see
Section~\ref{sec:qsprocs} for full details.

Evans and Hudson worked with bounded $\phi$ for an arbitrary algebra
$\alg$ and with finite-dimensional multiplicity space $\mul$. Their work was extended to
the case where $\mul$ is separable and infinite dimensional, $\alg$ is
a von Neumann algebra and $\phi$ is completely bounded in \cite{DaS92}
and \cite{GLW01}. A more thorough analysis was later given in
\cite{BLS12}, \cite{BLS13} and \cite{BeW14}, where $\alg$ is still a
von Neumann algebra and either $\phi$ is completely bounded or $\ini$
and $\mul$ are separable. This included an algebraic characterisation
of solutions to~(\ref{eqn:multeq}).

Here, we move to the broader setting of $C^*$~algebras. The technical
difficulties faced are much greater than before. Above, we were vague
about the nature of the tensor-product symbol. For those prior works
where $\alg$ is a von Neumann algebra it will be the von Neumann tensor
product and all bounded maps will be required to be normal. For processes on
$C^*$ algebras this must be modified; we have to use matrix-space
tensor products such as $\alg \matten \bop{ \mmul }{}$, and these need
not produce algebras \cite{LiW01}. Moreover, the matrix-space liftings of
$*$-homomorphisms such as $j_t \matten \id_{\bop{\mmul}{}}$ need not
preserve products. Thus much greater care is needed as, for example, 
checking equalities for simple tensors and extending by linearity and
continuity will no longer suffice.

This paper is structured as follows. Section~\ref{sec:liftings}
contains the necessary definitions and results concerning matrix
spaces. Section~\ref{sec:qsprocs} introduces and gives general results
about quantum stochastic processes, cocycles and the
QSDE~(\ref{eqn:EHeqn}). Here we work with processes on a general
operator space~$\opsp$; this point of view was used to unify several
results about various forms of cocycles in \cite{LiW14} and
\cite{LiW21}. The key result is Theorem~\ref{thm:intcocycle} which
gives conditions under which weak solutions of~(\ref{eqn:EHeqn}) are
actually cocycles. Section~\ref{sec:FK} contains our main results:
after introducing the notion of a free flow $j$, now on a $C^*$
algebra, and solving the corresponding multiplier equation, we
establish the effects of Feynman--Kac perturbation. We have given the
result in a very general form, necessitating a long list of hypotheses
regarding multiplicativity of liftings and measurability of processes,
but also provide a number of auxiliary results that can be used to
simplify matters greatly under conditions that frequently arise in
practice. In Section~\ref{sec:egs}, these results are applied to a
wide variety of examples. Our previous work on the quantum exclusion
process from \cite{BeW14} is extended; there we obtained the growth
estimates by assuming a symmetry condition on the amplitudes, which
can now be relaxed by means of the perturbation techniques of this
paper. (An alternative construction of exclusion processes
has been given in \cite{LiW21} by making use of the semigroup
characterisation of completely positive and contractive cocycles given
in \cite{LiW14} that generalised earlier work of Accardi and Kozyrev
\cite{AcK01}.) Similarly our previous work on flows on universal $C^*$
algebras from \cite{BeW14} is extended; Feynman--Kac perturbations
had already been used in this context for flows on the non-commutative
torus in \cite{CGS03}. Here we show that those techniques apply more
generally for other algebras such as the Cuntz algebras and the
non-commutative spheres of Banica \cite{Ban15}.

\subsection{Conventions and notation}

The indicator function of a set $A$ is denoted $1_A$ (with its domain being clear from the context)
whereas $1_\alg$ is the multiplicative identity for the unital $C^*$~algebra $\alg$
and $\tfn{P}$ equals $1$ if the proposition $P$ is true and $0$ if it is false.
All vector spaces have complex scalar
field; all inner products are linear in their second argument. The
identity operator on a vector space~$X$ is denoted $I_X$ whereas the
identity operator on an operator space~$\opsp$ is denoted
$\id_\opsp$. The Banach algebra of bounded operators on a Banach space~$X$
is denoted~$\bop{X}{}$; the $C^*$~algebra of $n \times n$ matrices
with entries from a $C^*$~algebra~$\alg$ is denoted~$M_n( \alg )$.
Algebraic, spatial and ultraweak tensor products are denoted
$\algten$, $\otimes$ and~$\uwkten$, respectively. The sets of
non-negative real numbers and integers are
denoted~$\R_+ := [ 0, \infty )$ and~$\Z_+ := \{ 0, 1, 2, \ldots \}$;
the set of natural numbers is denoted $\N := \{ 1, 2, 3, \ldots \}$.

\section{Matrix spaces and liftings}\label{sec:liftings}

\begin{definition}
Given Hilbert spaces $\ini$ and $\hilb$ and a vector $z \in \hilb$, let $E_z \in
\bop{\ini}{\ini \otimes \hilb}$ be such that $E_z u = u \otimes z$ for
all $u \in \ini$ and let
$E^z := E_z^* \in \bop{\ini \otimes \hilb}{\ini}$ be its adjoint. Using
Dirac notation, we may write $E_z = I_\ini \otimes \ket{z}$ and
$E^z = I_\ini \otimes \bra{z}$.
\end{definition}

\begin{definition}
Let $\opsp \subseteq \bop{\ini}{}$ be a concrete operator space, that
is, a norm-closed linear subspace of $\bop{\ini}{}$. Given any Hilbert
space $\hilb$, the \emph{matrix space over $\opsp$},
\[
\opsp \matten \bop{\hilb}{} := %
\{ T \in \bop{\ini \otimes \hilb}{} : E^z T E_w \in \opsp %
\text{ for all } z, w \in \hilb \},
\]
is an operator space such that
$\opsp \otimes \bop{\hilb}{} \subseteq \opsp \matten \bop{\hilb}{}
\subseteq \opsp \uwkten \bop{\hilb}{}$. The first inclusion is an
equality if $\hilb$ is finite dimensional; the latter inclusion is an
equality if and only if $\opsp$ is ultraweakly closed.
If $\hilc$ is another Hilbert space  then
$\bigl( \opsp \matten \bop{\hilb}{} \bigr) \matten \bop{\hilc}{} = %
\opsp \matten \bop{\hilb \otimes \hilc}{}$ with the standard
identifications.

Rectangular matrix spaces of the form $\opsp \matten
\bop{\hilb}{\hilc}$ are defined in the same way. Two important
examples are the \emph{column space}
\[
\opsp \matten \ket{\hilb} := %
\{ T \in \bop{\ini}{\ini \otimes \hilb} %
: E^z T \in \opsp \text{ for all } z \in \hilb \} %
= \opsp \matten \bop{\C}{\hilb}
\]
and the \emph{row space}
\[
\opsp \matten \bra{\hilb} := %
\{ T \in \bop{\ini \otimes \hilb}{\ini} %
: T E_w \in \opsp \text{ for all } w \in \hilb \} %
= \opsp \matten \bop{\hilb}{\C}.
\]
\end{definition}

\begin{definition}
If $\Phi : \opsp \to \opsq$ is a completely bounded map between
operator spaces and $\hilb$ is a Hilbert space then the
\emph{matrix-space lifting} of $\Phi$ is the unique completely bounded
map
\[
\Phi \matten \id_{\bop{\hilb}{}} : \opsp \matten \bop{\hilb}{} \to %
\opsq \matten \bop{\hilb}{}
\]
such that
$E^z ( \Phi \matten \id_{\bop{\hilb}{}} )( T ) E_w = %
\Phi( E^z T E_w )$
for all $T \in \opsp \matten \bop{\hilb}{}$ and $z$,
$w \in \hilb$. It then holds that
$\| \Phi \matten \id_{\bop{\hilb}{}} \|_\cb = \| \Phi \|_\cb$.
For the existence of such a map, see \cite{LiW01} or
\cite[Theorem~2.5]{Blt10}.

If $\hilc$ is another Hilbert space then, with the standard
identifications,
\begin{equation}\label{eqn:liftident}
( \Phi \matten \id_{\bop{\hilb}{}} ) \matten \id_{\bop{\hilc}{}} = %
\Phi \matten \id_{\bop{\hilb \otimes \hilc}{}}.
\end{equation}
\end{definition}

\begin{lemma}\label{lem:cplifting}
If $\Phi : \alg \to \blg$ is a completely positive map between unital
$C^*$~algebras and~$\hilc$ is a Hilbert space then the matrix-space
lifting $\Phi \matten \id_{\bop{\hilc}{}}$ is a completely positive
map between operator systems.
\end{lemma}
\begin{proof}
Assume without loss of generality that
$\blg \subseteq \bop{\hilb}{}$ and let
$T \in \alg \matten \bop{\hilc}{}$ be positive. To prove that
$( \Phi \matten \id_{\bop{\hilc}{}} )( T )$ is positive, it
suffices to show that
\[
\langle ( I_\hilb \otimes P ) \xi, %
( \Phi \matten \id_{\bop{\hilc}{}} )( T ) %
( I_\hilb \otimes P ) \xi \rangle \ge 0 %
\qquad ( \xi \in \hilb \otimes \hilc )
\]
for any finite-rank orthogonal projection $P \in \bop{\hilc}{}$. If
$P$ has rank $n$ then $( I_\hilb \otimes P ) \xi$ and $P T P$ may be
considered to be elements of $\hilb^n$ and $M_n( \alg )$, respectively, with the
latter being positive, so that
\[
\langle ( I_\hilb \otimes P ) \xi, %
( \Phi \matten \id_{\bop{\hilc}{}} )( T ) %
( I_\hilb \otimes P ) \xi \rangle = \langle ( I_\hilb \otimes P ) \xi, %
( \Phi \matten \id_{\C^n} )( P T P ) %
( I_\hilb \otimes P ) \xi \rangle \ge 0,
\]
as required. Since $\hilc$ is arbitrary, complete positivity now
follows by~(\ref{eqn:liftident}).
\end{proof}

\section{Quantum stochastic processes, cocycles and differential
equations}\label{sec:qsprocs}

\subsection{Stochastic processes}

\begin{notation}
Fix a Hilbert space $\mul$, the \emph{multiplicity space}. For any
subinterval $J$ of $\R_+$, let $\fock_J$ denote the Boson Fock space
over $L^2( J; \mul )$. The exponential vector corresponding to
$g \in L^2( J; \mul)$ is
$\evec{g} = \bigl( ( n! )^{-1/2} g^{\otimes n}\bigr)_{n \in \Z_+}$.
For brevity, let $\fock := \fock_{\R_+}$,
$\fock_{t)} := \fock_{[ 0, t )}$ and
$\fock_{[t} := \fock_{[ t, \infty )}$ for all~$t \in \R_+$, with
similar abbreviations for the identity operators $I$, $I_{t)}$ and
$I_{[t}$ on these spaces, and where $\fock_{0)} := \C$. Recall the
tensor-product decomposition
\[
\fock \cong \fock_{t)} \otimes \fock_{[t}; \ %
\nvec{f} \mathbin{\leftrightarrow} %
\nvec{f|_{[ 0, t )}} \otimes \nvec{f|_{[ t, \infty )}} \qquad %
\text{for all } t \in \R_+, f \in \elltwo,
\]
where $\nvec{g} = \exp( -\hlf \| g \|^2 ) \evec{g}$ is the normalised
exponential vector corresponding to $g$. This identification will be used frequently
without comment.

The \emph{extended multiplicity space} is $\mmul := \C \oplus \mul$,
and $\wh{z} := ( 1, z ) \in \mmul$ for all~$z \in \mul$.

Let $\Delta_\hilb := I_\hilb \otimes P_\mul \in \bop{\hilb \otimes
\mmul}{}$, where $P_\mul \in \bop{\mmul}{}$ is the orthogonal
projection onto $\{0\} \oplus \mul$. This will be abbreviated to
$\Delta$ when the Hilbert space $\hilb$ is clear from the context.
\end{notation}

\begin{definition}
An \emph{admissible set} is a subset of $\mul$ which contains $0$ and
is total in $\mul$. For such a set $\tot$, let
$\step = \lin \{ 1_{[0,t)} x : x \in \tot, t > 0 \}$ denote the set
of $\tot$-valued right-continuous step functions on
$\R_+$. Admissibility of $\tot$ ensures that $\evecs( \tot )$, the
linear span of the exponential vectors corresponding to
elements of $\step$, is dense in~$\fock$
\cite[Proposition~2.1]{Lin05}.
\end{definition}

\begin{definition}\label{dfn:opproc}
Fix a Hilbert space $\ini$, the \emph{initial space}. For a fixed
admissible set $\tot$, an \emph{operator process} $X$ is a collection
of linear operators $( X_t )_{t \in \R_+}$ such that
\begin{mylist}
\item[(i)]
$\ini \algten \evecs( \tot ) \subseteq \dom X_t %
\subseteq \ini \otimes \fock$ and
$\im X_t \subseteq \ini \otimes \fock$ for all $t \in \R_+$,
\item[(ii)] adaptedness holds, in the sense that
\[
E^{\nvec{f}} X_t E_{\nvec{g}} = \langle \nvec{1_{[ t, \infty )} f}, %
\nvec{1_{[ t, \infty )} g} \rangle %
E^{\nvec{1_{[ 0, t )} f}} X_t E_{\nvec{1_{[ 0, t )} g}}
\]
for all $f$, $g \in \step$ and $t \in \R_+$, and
\item[(iii)] $t \mapsto X_t \xi$ is weakly measurable
for all $\xi \ini \algten \evecs( \tot )$.
\end{mylist}
The operator process $X$ is \emph{strongly measurable} or
\emph{strongly continuous} if $t \mapsto X_t \xi$ is strongly
measurable or norm continuous, respectively, for all $\xi \in \ini
\algten \evecs( \tot )$; it is \emph{weakly continuous} if $t \mapsto
\langle \zeta, X_t \xi \rangle$ for all $\zeta$, $\xi \in \ini \algten
\evecs( \tot )$. The operator process $X$ is \emph{bounded},
\emph{contractive}, \emph{isometric}, \emph{co-isometric} or
\emph{unitary} if each operator $X_t$ has this property. In these
cases we will automatically identify $X_t$ with its continuous
extension to all of $\ini \otimes \fock$, and then adaptedness means
that $X_t = X_{t)} \otimes I_{[t,\infty)}$ for some
$X_{t)} \in \bop{\ini \otimes \fock_{t)}}{}$.
\end{definition}

\begin{remark}\label{rmk:procprops}
\begin{mylist}
\item[(i)] A bounded operator process which is strongly continuous has
locally bounded norm if and only if it is strongly continuous on
all of $\ini \otimes \fock$, by the Banach--Steinhaus Theorem.
\item[(ii)] If $X$ and $Y$ are strongly continuous bounded operator
processes that each have locally bounded norm then the product
$X Y = ( X_t Y_t )_{t \in \R_+}$ is also a strongly continuous
operator process with locally bounded norm.
\item[(iii)] If $\ini$ and $\mul$ are both separable then all operator
processes are strongly measurable, by Pettis' Theorem.
\end{mylist}
\end{remark}

\begin{definition}\label{dfn:cbmp}
A \emph{completely bounded mapping process} $k$ on an operator space
$\opsp \subseteq \bop{\ini}{}$ is a collection of completely bounded maps
$\bigl( k_t : %
\opsp \to \bop{\ini \otimes \fock}{} \bigr)_{t \in \R_+}$ such that
$\bigl( (k_t( x ) \bigr)_{t \in \R_+}$ is an operator process for each
$x \in \opsp$. The adaptedness of each of these processes implies for
each $t \in \R_+$ the existence of a completely bounded map
$k_{t)} : \opsp \to \bop{\ini \otimes \fock_{t)}}{}$ such that
$k_t( x ) = k_{t)}( x ) \otimes I_{[t}$ for all~$x \in \opsp$.

If $k$ is a completely bounded mapping process on $\opsp$ then $k$ is
\emph{strongly measurable} or \emph{strongly continuous} if the
bounded operator process $\bigl( k_t( x ) \bigr)_{t \in \R_+}$ has the
same property for all~$x \in \opsp$. If~$\opsp$ is a $*$-algebra then
the mapping process $k$ is \emph{$*$-homomorphic} if $k_t$ is a
$*$-homomorphism for each $t \in \R_+$.
\end{definition}

\subsection{Cocycles}

\begin{definition}\label{dfn:cocycle}
Given a completely bounded mapping process $k$ on $\opsp$, let
\[
k_t[ f, g ] : \opsp \to \bop{\ini}{}; \ x \mapsto %
E^{\nvec{1_{[ 0, t )} f}} k_t( x ) E_{\nvec{1_{[ 0, t )} g}} = %
E^{\nvec{f|_{[ 0, t )}}} k_{t)} ( x ) E_{\nvec{g|_{[ 0, t )}}}
\]
for all $t \in \R_+$ and $f$,~$g \in \elltwo$. A
\emph{Markovian cocycle on $\opsp$} is a completely bounded mapping
process $k$ such that, for some admissible set $\tot$,
\begin{mylist}
\item[(i)] $k_t[ f, g ]( x ) \in \opsp$ for all $t \in \R_+$, $f$,
$g \in \step$ and $x \in \opsp$,
\item[(ii)] $k_0[ 0, 0 ] = \id_\opsp$ and
\item[(iii)] $k_{s + t}[ f, g ] = %
k_s[ f, g ] \comp k_t[ f( \cdot + s ), g( \cdot + s ) ]$ for all $s$,
$t \in \R_+$ and $f$, $g \in \step$.
\end{mylist}
Since $k$ is adapted, if~(ii) holds then $k_0[ f,g ] = \id_\opsp$ for
all $f$, $g$ in $\step$.
\end{definition}

\begin{remark}\label{rem:markovian}
Totality of the set $\{ \nvec{f} : f \in \step \}$, linearity and norm
continuity of the map $z \mapsto E_z$, and adaptedness of $k$, show
that condition~(i) of Definition~\ref{dfn:cocycle} is equivalent to
the requirement that
\begin{mylist}
\item[(i)$'$] $k_t( \opsp ) \subseteq \opsp \matten \bop{\fock}{}$.
\end{mylist}
Furthermore, in this formulation, the following equivalent
versions of conditions~(ii) and~(iii) appear as well:
\begin{mylist}
\item[(ii)$'$] $k_0( x ) = x \otimes I_\fock$ for all $x \in \opsp$ and
\item[(iii)$'$] $k_{s+t} = \wh{k}_s \comp \sigma_s \comp k_t$ for all $s, t \in \R_+$,
\end{mylist}
where $\sigma_s : \bop{\ini \otimes \fock}{} \to %
\bop{\ini \otimes \fock_{[s}}{}$ is the ampliated right shift, arising
from the natural unitary identification
$\ini \otimes \fock \cong \ini \otimes \fock_{[s}$. In particular, if
$T \in \bop{\ini \otimes \fock}{}$ then
\[
E^{\nvec{f}} \sigma_s( T ) E_{\nvec{g}} = %
E^{\nvec{f( \cdot + s )}} T E_{\nvec{g( \cdot + s )}} %
\qquad \text{for all } f, g \in L^2\bigl( [ s, \infty ); \mul \big),
\]
and $\wh{k}_s := k_{s)} \matten \id_{\bop{\fock_{[s}}{}}$; note the
equality and inclusion
\[
\sigma_s\bigl( \opsp \matten \bop{\fock}{} \bigr) = %
\opsp \matten \bop{\fock_{[s}}{} \qquad \mbox{and} \qquad
k_{s)}( \opsp ) \subseteq \opsp \matten \bop{\fock_{s)}}{},
\]
where the latter follows from condition~(i). The definition of
$\sigma_s$ here differs slightly from that used in~\cite{LiW14}.

Thus whether a completely bounded mapping process is a Markovian
cocycle is independent of the choice of admissible set $\tot$, and the
set $\step$ in Definition~\ref{dfn:cocycle} can be replaced by any
subset $S$ of $\elltwo$ for which $\{ \nvec{f} : f \in S \}$ is total
in $\fock$.
\end{remark}

The following proposition is the matrix-space version of the
corresponding result for normal mapping processes on von~Neumann
algebras \cite[Lemma~1.6(b)]{BLS13}.

\begin{proposition}\label{prp:cosemi}
Let $k$ be a completely bounded mapping process on $\opsp$. Then $k$
is a Markovian cocycle if and only if
$( K_t := \wh{k}_t \comp \sigma_t )_{t \in \R_+}$ is a semigroup
on~$\opsp \matten \bop{\fock}{}$, that is,
\[
K_t\bigl( \opsp \matten \bop{\fock}{} \bigr) \subseteq %
\opsp \matten \bop{\fock}{}, \quad %
K_0 = \id_{\opsp \matten \bop{\fock}{}} \quad \text{and} \quad %
K_{s + t} = K_s \comp K_t \quad \text{for all } s, t \in \R_+.
\]
\end{proposition}
\begin{proof}
Note first that $\dom K_t = \opsp \matten \bop{\fock}{}$, and that
\[
k_t( x ) = K_t( x \otimes I ) %
\qquad \text{for all } t \in \R_+ \text{ and } x \in \opsp;
\]
it follows immediately that if $K$ is a semigroup on
$\opsp \matten \bop{\fock}{}$ then $k$ is a Markovian cocycle.

For the converse, let $k$ be a completely bounded mapping process on
$\opsp$. The identity
\[
E^{\nvec{f}} K_s( T ) E_{\nvec{g}} = k_s[ f, g ] \bigl( E^{\nvec{f (
\cdot + s)}} T E_{\nvec{g ( \cdot + s) }} \bigr)
\]
holds for all $f, g \in \step$, $s \in \R_+$ and
$T \in \opsp \matten \bop{\fock}{}$. Hence if $k$ is a Markovian
cocycle then
$K_s \bigl( \opsp \matten \bop{\fock}{} \bigr) \subseteq \opsp \matten
\bop{\fock}{}$. Moreover, this identity and conditions~(ii) and~(iii)
of Definition~\ref{dfn:cocycle} then give that
\[
K_0( T ) = T \qquad \text{and} \qquad %
K_s \bigl( K_t( T ) \bigr) = K_{s+t}( T )
\]
for all $s$, $t \in \R_+$ as required.
\end{proof}

\subsection{Semigroup representation}

\begin{notation}
Given a completely bounded mapping process $k$ on $\opsp$, let
\[
\cP^{z, w}_t := k_t[ 1_{[ 0, t )} z, 1_{[ 0, t )} w ] : %
\opsp \to \bop{\ini}{}; \ x \mapsto %
E^{\nvec{1_{[ 0, t )} z}} k_t( x ) E_{\nvec{1_{[ 0, t )} w}}
\]
for all $z$, $w \in \mul$ and $t \in \R_+$.
\end{notation}

\begin{theorem}\label{thm:cocycle}
A completely bounded mapping process $k$ on $\opsp$ is a Markovian
cocycle if and only if there exists an admissible set $\tot$ such that
$( \cP^{z, w}_t )_{t \in \R_+}$ is a semigroup on $\opsp$ for all $z$,
$w \in \tot$ and, given any $f$, $g \in \step$ subordinate to
some partition $\{ 0 = t_0 < \cdots < t_n < \cdots \}$ of~$\R_+$, it
holds that
\begin{equation}\label{eqn:cocyclesemigroup}
k_t[ f, g ] = \cP^{z_0, w_0}_{t_1 - t_0} \comp \cdots \comp %
\cP^{z_n, w_n}_{t - t_n} \qquad \text{whenever } %
t \in [ t_n, t_{n + 1} ),
\end{equation}
where
$z_j = f( t_j )$ and $w_j = g( t_j )$ for~$j = 0$, \ldots, $n$. In
this case the decomposition~(\ref{eqn:cocyclesemigroup}) holds for any
choice of admissible set $\tot$.
\end{theorem}
\begin{proof}
See \cite[Proposition~5.1]{LiW14}.
\end{proof}

\begin{remark}
Let $k$ be a completely bounded mapping process on $\opsp$ and let
$\tot$ be an admissible set. If
$\{ \cP^{z, w}_t : z, w \in \tot, \ t \in \R_+ \}$ is a family of
linear maps on $\opsp$ that satisfies (\ref{eqn:cocyclesemigroup})
for all functions $f$, $g \in \step$ that are subordinate to the partition
$\{ 0 = t_0 < \cdots < t_n < \cdots \}$ then
\[
\cP^{z, w}_s \comp \cP^{z, w}_t = \cP^{z, w}_{s + t} \qquad %
\textrm{for all } z, w \in \tot \text{ and } s, t \in \R_+.
\]
To see this, let $f$ and $g$
take the values~$z$ and $w$, respectively, on a sufficiently large
interval containing the origin. Thus, it is not necessary to verify
independently that the semigroup property holds when applying
Theorem~\ref{thm:cocycle} to produce a Markovian cocycle.
\end{remark}

\subsection{Differential equations}

\begin{notation}
Let $\opsp_0$ be a subset of $\opsp$ and let
$\psi : \opsp_0 \to \opsp \matten \bop{\mmul}{}$ be a map.  Given a
completely bounded mapping process $k$ on $\opsp$ define
$\wt{k}_t := k_t \matten \id_{\bop{\mmul}{}}$. For a fixed admissible
set~$\tot$, the statement that
\[
\rd k_t( x ) = ( \wt{k}_t \comp \psi )( x ) \std \Lambda_t \qquad %
\text{weakly on } \ini \algten \evecs( \tot )
\]
for some $x \in \opsp_0$ means that
\[
t \mapsto \langle u \otimes \nvec{f}, k_t\bigl( E^{\wh{f( t )}} %
\psi( x ) E_{\wh{g( t )}} \bigr) v \otimes \nvec{g} \rangle
\]
is locally integrable and
\[
\langle u \otimes \nvec{f}, %
\bigl( k_t( x ) - k_0( x ) \bigr) v \otimes \nvec{g} \rangle = %
\int_0^t \langle u \otimes \nvec{f}, k_s\bigl( E^{\wh{f( s )}} %
\psi( x ) E_{\wh{g( s )}} \bigr) v \otimes \nvec{g} \rangle \std s
\]
for all $u$, $v \in \ini$, $f$, $g \in \step$ and $t \in \R_+$. If
this holds for all $x \in \opsp_0$ then $k$ is a \emph{weak solution}
of the Evans--Hudson QSDE in the terminology of
\cite{Lin05, LiP98,LiW00a}.
A \emph{strong solution} must, in addition, satisfy the
extra requirement that $\bigl( \wt{k}_t( \psi( x ) ) \bigr)_{t \in \R_+}$
be integrable, as explained below in
Definition~\ref{dfn:integrable}.
\end{notation}

\begin{remark}
Recall that a \emph{$C_0$~semigroup} on the Banach space $X$ is a
family $( T_t )_{t \in \R_+} \subseteq \bop{X}{}$ such that
\[
T_0 = I_X, \qquad %
T_{s + t} = T_s \comp T_s \quad \text{for all } s, t \in \R_+ \quad %
\text{and} \quad \lim_{t \to 0} \| T_t( x ) - x \| = 0 \quad %
\text{for all } x \in X.
\]
It follows from these assumptions that the map
\[
\R_+ \times X \to X; \ ( t, x ) \mapsto T_t( x )
\]
is jointly continuous \cite[Theorem~6.2.1]{Dav07}.

Any $C_0$ semigroup is characterised by its \emph{infinitesimal
generator} $\tau : \dom \tau \to X$, the closed, densely defined
operator such that
\[
\dom \tau = \bigl\{ x \in X : \lim_{t \to 0} t^{-1} ( T_t( x ) - x ) %
\text{ exists} \bigr\} \quad \text{and} \quad %
\tau( x ) = \lim_{t \to 0} t^{-1} ( T_t( x ) - x ).
\]
Note that $\dom \tau$ is left invariant by $T_t$
\cite[Lemma~6.1.11]{Dav07}. A \emph{core} for $\tau$ is a subspace
$X_0$ of $\dom \tau$ such that the graph of the restriction
$\tau|_{X_0}$ is dense in the graph of $\tau$, regarded as a subset of
the normed space $X \oplus X$ with the subspace topology.
\end{remark}

\begin{notation}\label{not:chi}
Let $\chi: \mul \times \mul \to \C$ denote the map $\chi( z, w ) =
\hlf \| z \|^2 + \hlf \| w \|^2 - \langle z, w \rangle$.
\end{notation}

\begin{theorem}\label{thm:intcocycle}
Let $k$ be a completely bounded mapping process on $\opsp$ with
locally bounded norm. Suppose there exists an admissible set $\tot$, a
norm-dense subspace $\opsp_0$ of $\opsp$ and a linear map
$\psi : \opsp_0 \to \opsp \matten \bop{\mmul}{}$ such that
\[
k_0( x ) = x \otimes I \qquad \text{and} \qquad \rd k_t( x ) = %
( \wt{k}_t \comp \psi ) ( x ) \std \Lambda_t \quad %
\text{ weakly on } \ini \algten \evecs( \tot )
\]
for all $x \in \opsp_0$. If, for all $z$, $w \in \tot$,
there exists a $C_0$ semigroup generator
$\eta_{z, w} : \dom \eta_{z, w} \to \opsp$ with a core
$\opsp_{z, w} \subseteq \opsp_0$ such that
$\eta_{z, w}( x ) = E^{\wh{z}} \psi( x ) E_{\wh{w}}$ for all
$x \in \opsp_{z, w}$ then $k$ is a Markovian cocycle.
\end{theorem}
\begin{proof}
If $f$, $g \in \step$ are constant on $[ s, t ) \subseteq \R_+$, say
$f( s ) = z$ and $g( s ) = w$, then
\[
\langle u \otimes \nvec{f}, %
\bigl( k_t( x ) - k_s( x ) \bigr) v \otimes \nvec{g} \rangle = %
\int_s^t \langle u \otimes \nvec{f}, %
k_r\bigl( \eta_{z, w}( x ) \bigr) v \otimes \nvec{g} \rangle \std r
\]
for all $u$, $v \in \ini$ and $x \in \opsp_{z, w}$. Given
$y \in \dom \eta_{z, w}$, there exists a sequence
$( x_n )_{n \in \N} \subseteq \opsp_{z, w}$ such that
\[
\| x_n - y \| + \| \eta_{z, w}( x_n - y ) \| \to 0 %
\qquad \text{as } n \to \infty
\]
and therefore Lebesgue's Dominated Convergence Theorem implies that,
in the weak sense,
\[
k_t( f, g )( y ) - k_s( f, g )( y ) = %
\int_s^t k_r( f, g )\bigl( \eta_{z, w}( y ) \bigr) \std r %
\qquad \text{for all } y \in \dom \eta_{z, w},
\]
where $k_r( f, g )( y ) := E^{\nvec{f}} k_r( y ) E_{\nvec{g}}$ and so
on. In particular, the function $r \mapsto k_r( f, g )( y )$ is
weakly continuous on $[ s, t ]$ for all $y \in \dom \eta_{z, w}$, so
for all $y \in \opsp$.

We now follow \cite[Proof of Proposition~3.4]{LiP98}. Let $f$,
$g \in \step$ be subordinate to the partition
$\{ 0 = t_0 < t_1 < \cdots \}$, let
$t \in ( t_n, t_{n + 1} ]$ for some $n \in \Z_+$,
let~$y \in \dom \eta_{z, w}$, let
$u$, $v \in \ini$ and consider the function
\[
F : [ t_n, t ] \to \C; \ s \mapsto \langle u, %
k_s( f, g )\bigl( \cQ^{z, w}_{t - s}( y ) \bigr) v \rangle,
\]
where $( \cQ^{z, w}_t )_{t \in \R_+}$ is the $C_0$~semigroup on
$\opsp$ generated by $\eta_{z, w}$, with $z = f( t_n )$ and
$w = g( t_n )$. The weak continuity of $r \mapsto k_r( f, g )( x )$
for all $x \in \opsp$ and the norm continuity of
$r \mapsto \cQ^{z, w}_r( y )$, together with the locally uniform
boundedness of the norm of $k$, give the continuity of~$F$.

If $s \in [ t_n, t )$ and $h \in ( 0, t - s )$ then, since
$\dom \eta_{z, w}$ is invariant under $\cQ^{z, w}$,
\begin{align*}
h^{-1} \bigl( F( s + h ) - F( s ) \bigr) & = %
\langle u, ( k_{s + h} - k_s )( f, g )%
\bigl( h^{-1} ( \cQ^{z, w}_{t - s - h} - \cQ^{z, w}_{t - s} )( y ) \bigr) v %
\rangle \\
& \qquad + h^{-1} \langle u, ( k_{s + h} - k_s )( f, g )%
\bigl( \cQ^{z, w}_{t - s}( y ) \bigr) v \rangle \\[1ex]
 & \qquad + \langle u, k_s( f, g )\bigl( h^{-1} %
\bigl( \cQ^{z, w}_{t - s - h} - \cQ^{z, w}_{t - s} \bigr) ( y ) %
\bigr) v \rangle \\[1ex]
 & \to 0 + \langle u, k_s( f, g )\bigl( %
\eta_{z, w}( \cQ^{z, w}_{t - s}( y ) )\bigr) v \rangle - %
\langle u, k_s( f, g )%
\bigl( \eta_{z, w}( \cQ^{z, w}_{t - s}( y ) ) \bigr) v \rangle = 0
\end{align*}
as $h \to 0$, so $F$ is constant on $[ t_n, t ]$, being a continuous
function with vanishing right derivative. Hence
\[
k_t( f, g )( y ) = %
\bigl( k_{t_n}( f, g ) \comp \cQ^{z, w}_{t - t_n} \bigr)( y ) %
\qquad \text{for all } y \in \dom \eta_{z, w},
\]
and so for all $y \in \opsp$. Repeating this argument gives that
\begin{equation}\label{eqn:semi1}
k_t( f, g ) = \langle \nvec{f}, \nvec{g} \rangle \, %
\cQ^{z_0, w_0}_{t_1 - t_0} \comp \cdots \comp %
\cQ^{z_n, w_n}_{t - t_n},
\end{equation}
where now $f( t_j ) = z_j$ and $g( t_j ) = w_j$ for $j = 0$, \ldots,
$n$. In particular, for any $z, w \in \tot$,
\[
\cP^{z, w}_s = k_s[ 1_{[0,s)} z, 1_{[0,s)} w ] = %
k_s( 1_{[0,s)} z, 1_{[0,s)} w ) = \exp( -s \chi( z, w ) ) \cQ^{z, w}_s.
\]
This relationship between $\cP^{z, w}$ and $\cQ^{z, w}$, together
with~(\ref{eqn:semi1}), yields
\[
k_t[ f, g ] = \cP^{z_0, w_0}_{t_1 - t_0} %
\comp \cdots \comp \cP^{z_n, w_n}_{t - t_n},
\]
and therefore $k$ is a Markovian cocycle, by
Theorem~\ref{thm:cocycle}.
\end{proof}

\begin{proposition}\label{prp:assocsemigen}
Let $k$ be a Markovian cocycle on $\opsp$ with locally bounded
norm. Suppose there exists an admissible set $\tot$, a norm-dense
subspace $\opsp_0$ of $\opsp$ and a linear map $\psi : \opsp_0 \to
\opsp \matten \bop{\mmul}{}$ such that
\begin{equation}\label{eqn:EHweak}
k_0( x ) = x \otimes I \qquad \text{and} \qquad \rd k_t( x ) = %
( \wt{k}_t \comp \psi ) ( x ) \std \Lambda_t \quad %
\text{ weakly on } \ini \algten \evecs( \tot )
\end{equation}
for all $x \in \opsp_0$. If one of the associated semigroups
$\cP^{z,w}$ of the cocycle is $C_0$ for some choice of $z,w \in \mul$
then all of them are $C_0$ semigroups. Moreover
$\psi^{\wh{z}}_{\wh{w}} := E^{\wh{z}} \psi( \cdot ) E_{\wh{w}} : %
\opsp_0 \to \opsp$ is closable for all $z, w \in \tot$ and the
generator of $\cP^{z,w}$ is an extension of the map
$x \mapsto \psi^{\wh{z}}_{\wh{w}}( x ) -\chi( z, w ) x$.
\end{proposition}
\begin{proof}
Let $t > 0$ and $x \in \opsp_0$. Since $k$ satisfies the
QSDE~(\ref{eqn:EHweak}) weakly, it follows that
\begin{equation}\label{eqn:assocsemigp}
t^{-1} \bigl( \cP^{z,w}_t( x ) -x ) = t^{-1} \biggl( \int^t_0 e^{(s-t)
\chi( z, w )} \cP^{z, w}_s \bigl( \psi^{\wh{z}}_{\wh{w}}( x )
\bigr) \std s +(e^{-t \chi( z, w )} -1) x \biggr)
\end{equation}
in the weak operator sense. However it is known
\cite[Proposition~5.4]{LiW00b} that either all or none of the
associated semigroups are $C_0$; if they do all have this property
then the integrand is a norm-continuous map from $\R_+$ to $\opsp$, so
equation~(\ref{eqn:assocsemigp}) in fact makes sense as a Bochner
integral. Note then that the limit of the right-hand side exists as
$t \to 0$, showing that $\opsp_0$ is contained in the domain of
the generator of $\cP^{z, w}$, whose action on $\opsp_0$ is as claimed.
\end{proof}

\section{Feynman--Kac perturbation}\label{sec:FK}

\subsection{The free flow}

\begin{definition}\label{dfn:qsflow}
A \emph{quantum stochastic flow} is a Markovian cocycle $j$
on a unital $C^*$~algebra $\alg \subseteq \bop{\ini}{}$ such that each
$j_t$ is a unital $*$-homomorphism and the mapping process $j$ is
strongly continuous.
\end{definition}

\begin{remark}
Note that in this paper we insist that a flow be a cocycle but make no
requirement that it solve a QSDE. This is at somewhat at odds with the
terminology adopted elsewhere.
\end{remark}

\begin{definition}\label{dfn:standard}
Let $j$ be a quantum stochastic flow on the unital $C^*$ algebra
$\alg$ and let $\tot$ be an admissible set. Suppose that
\begin{equation}\label{eqn:flowqsde}
\rd j_t( x ) = %
( \wt{\jmath}_t \comp \phi )( x ) \std \Lambda_t %
\qquad \text{weakly on } \ini \algten \evecs( \tot ) %
\qquad \text{for all } x \in \alg_0,
\end{equation}
where $\alg_0$ is a subspace of $\alg$ and
$\phi : \alg_0 \to \alg \matten \bop{\mmul}{}$. The map $\phi$, called
the \emph{generator} of the flow~$j$, has \emph{standard form}
if~$\alg_0$ is a norm-dense $*$-algebra containing the multiplicative
unit $1_\alg$ and $\phi$ is a linear map with the block-matrix
decomposition
\begin{equation}\label{eqn:blockgenform}
\phi = \begin{bmatrix} \tau & \delta^\dagger \\[1ex]
 \delta & \pi - \iota \end{bmatrix} : \alg_0 \to \begin{bmatrix}
 \alg & \alg \matten \bra{\mul} \\[1ex]
 \alg \matten \ket{\mul} & \alg \matten \bop{\mul}{}
\end{bmatrix},
\end{equation}
where
\begin{mylist}
\item[(i)]  $\tau : \alg_0 \to \alg$ is a $*$-linear map such that
\[
\tau( x y ) - \tau( x ) y - x \tau( y ) = \delta^\dagger( x ) %
\delta( y )
\qquad \text{for all } x, y \in \alg_0,
\]
\item[(ii)] $\delta : \alg_0 \to \alg \matten \ket{\mul}{}$ is a
$\pi$-derivation, so that
\[
\delta( x y ) = \delta( x ) y + \pi( x ) \delta( y ) %
\qquad \text{for all } x, y \in \alg_0,
\]
\item[(iii)] $\delta^\dagger : \alg_0 \to \alg \matten \bra{\mul}{}$
is such that $\delta^\dagger( x ) = \delta( x^* )^*$ for all
$x \in \alg_0$, and
\item[(iv)] 
$\pi : \alg_0 \to \alg \matten \bop{\mul}{}$ is a
unital $*$-homomorphism and
$\iota : \alg_0 \to \alg \matten \bop{\mul}{}$ is the ampliation map
$x \mapsto x \otimes I_\mul$.
\end{mylist}
\end{definition}

\begin{remark}
Suppose (\ref{eqn:flowqsde}) holds with $\alg_0 = \alg$.
\begin{mylist}
\item[(i)] If the multiplicity space $\mul$ is finite dimensional and
$\phi$ is bounded, then $j$ satisfies~(\ref{eqn:flowqsde}) strongly
and $\phi$ has standard form \cite[Theorem~5.1(d)]{LiP98}.

\item[(ii)] More generally, if there exists an orthonormal basis
$\{ e_\alpha : \alpha \in A \}$ of $\mmul$ that contains the vector $( 1, 0 )$
and is such that
\[
\alg \to \bop{\ini}{}; \ x \mapsto E^{e_\alpha} \phi( x ) E_{e_\beta}
\]
is a bounded map for all $\alpha$, $\beta \in A$ then $\phi$ is
completely bounded \cite[Proposition~5.1 and Theorem~5.2]{LiW00a} and
has standard form
\cite[Proposition~3.2, Proposition~6.3 and Theorem~6.5]{LiW00a}.
\end{mylist}
\end{remark}

\begin{example}\label{ex: gauge free+bdd}
Suppose $\alg$ is a unital $C^*$~algebra, let
$t \in \alg \matten \ket{\mul}$ be such that
$t^* ( x \otimes I_\mul) t \in \alg$ for all $x \in \alg$
and let $h = h^* \in \alg$. The map
\begin{equation}\label{eqn: gauge free block}
\phi = \begin{bmatrix}
 \tau & \delta^\dagger \\
 \delta & 0
\end{bmatrix}
\end{equation}
has standard form, where 
\begin{align*}
\delta( x ) & = ( x \otimes I_\mul ) t - t x \\[1ex]
\text{and} \quad \tau( x ) & = \I [ h, x ] - %
\hlf t^* t x + t^* ( x \otimes I_\mul ) t - \hlf x t^* t %
\qquad \textrm{for all } x \in \alg.
\end{align*}
\end{example}

\begin{example}\label{ex: gauge free skew derivs}
Let $\alg$ be a unital $C^*$~algebra
with a norm-dense $*$-subalgebra $\alg_0 \subseteq \alg$
containing the multiplicative unit
$1_\alg$. For $i = 1$, \ldots, $n$, let $c_i \in \C$ and let
$\delta_i : \alg_0 \to \alg$ be a \emph{skew-symmetric} derivation, so
that $\delta_i^\dagger = -\delta_i$.
With $\{e_1, \ldots, e_n\}$ the standard orthonormal basis of
$\mul = \C^n$,
\[
\delta( x ) := \sum_{i = 1}^n c_i \delta_i( x ) \otimes \ket{e_i} %
\qquad \text{and} \qquad %
\tau( x ) := -\frac{1}{2} \sum_{i = 1}^n |c_i|^2 \delta_i^2( x ) \qquad
\text{for all } x \in \alg_0,
\]
the map $\phi$ given by~(\ref{eqn: gauge free block}) has standard form.
\end{example}

\begin{remark}
The challenge is to extend these paradigmatic examples. The map~$\phi$ defined in Example~\ref{ex: gauge free+bdd} is completely
bounded, so the QSDE~(\ref{eqn:flowqsde}) can be solved \cite{LiW01}
and the solution is a strong one. It is natural to ask what happens when $\tau$
is no longer bounded, for instance when
the derivations $\delta_1$, \ldots, $\delta_n$ in
Example~\ref{ex: gauge free skew derivs} are unbounded. Can we still solve~(\ref{eqn:flowqsde})? These challenges are addressed in
Section~\ref{sec:egs}.
\end{remark}

\subsection{The multiplier equation}
\label{sec:multeqn}

\begin{definition}
A bounded operator process $X$ is a \emph{right multiplier cocycle}
for the quantum stochastic flow $j$ \cite[Definition~2.2]{BLS13} if
\[
X_{s + t} = J_s( X_t ) X_s \qquad \text{for all } s, t \in \R_+,
\]
where $J_s := \wh{\jmath}_s \comp \sigma_s$ as in
Proposition~\ref{prp:cosemi}.
\end{definition}

\begin{proposition}\label{prp:multiplier1}
Let $j$ be a quantum stochastic flow on the unital $C^*$~algebra
$\alg$. Suppose $k$ is a Markovian cocycle on $\alg$ and $Y$ is a
bounded operator process such that
\[
k_t( x ) = j_t( x ) Y_t %
\qquad \text{for all } t \in \R_+ \text{ and } x \in \alg.
\]
Then $Y$ is a right multiplier cocycle for the flow $j$.
\end{proposition}
\begin{proof}
Note first that
\[
k_{t)}( x ) = j_{t)}( x ) Y_{t)} \quad \text{and} \quad
Y_t E_{\nvec{g}} = E_{\nvec{g_{[t}}} Y_{t)} E_{\nvec{g_{t)}}}
\qquad \text{for all } t \in \R_+ \text{and} x \in \alg,
\]
using the notation of
Definitions~\ref{dfn:opproc} and~\ref{dfn:cbmp}. Consequently we have that
$\wh{k}_t( R ) = \wh{\jmath}_t( R ) Y_t$ for all
$R \in \alg \matten \bop{\fock_{[t}}{}$, thus
$K_t( T ) = J_t( T ) Y_t$ for all $t \in \R_+$ and
$T \in \alg \matten \bop{\fock}{}$ in the notation of
Remark~\ref{rem:markovian} and Proposition~\ref{prp:cosemi}. Hence
\[
Y_{s + t} = k_{s + t}( I_\ini ) = %
( \wh{k}_s \comp \sigma_s \comp k_t )( I_\ini ) = %
K_s( Y_t ) = J_s( Y_t ) Y_s \qquad \text{for all } s, t \in \R_+.
\qedhere
\]
\end{proof}

\begin{remark}
Proposition~\ref{prp:multiplier1} reverses the reasoning used in the
von~Neumann-algebraic context, where it is shown that $k$ is a
Markovian cocycle if $Y$ is a suitable right multiplier cocycle
\cite[Proposition~2.5]{BLS13}. The obstruction in the $C^*$~setting
appears when attempting to show that~$j_t( a ) Y_t$ lies
in~$\alg \matten \bop{\fock}{}$: see Remark~\ref{rmk:notalg}.
\end{remark}

\begin{definition}\label{dfn:integrable}
Given an admissible set $\tot$, an \emph{integrable process} $F$ is a
collection of linear operators $( F_t )_{t \in \R_+}$ such that
\begin{mylist}
\item[(i)]
$\ini \algten \evecs( \tot ) \algten \mmul \subseteq \dom F_t %
\subseteq \ini \otimes \fock \otimes \mmul$ and
$\im F_t \subseteq \ini \otimes \fock \otimes \mmul$
for all $t \in \R_+$,
\item[(ii)]
$E^{\nvec{f} \otimes \xi} F_t E_{\nvec{g} \otimes \eta} = %
\langle \nvec{1_{[ t, \infty )} f}, \nvec{1_{[ t, \infty )} g} \rangle %
E^{\nvec{1_{[ 0, t )} f} \otimes \xi} F_t %
E_{\nvec{1_{[ 0, t )} g} \otimes \eta}$
for almost all $t \in \R_+$, for all $f$, $g \in \step$ and
$\xi$, $\eta \in \mmul$,
\item[(iii)] $t \mapsto F_t \xi$ is strongly measurable for all
  $\xi \in \ini \algten \evecs( \tot ) \algten \mmul$,
\item[(iv)] $t \mapsto \Delta^\perp_{\ini \otimes \fock} F_t %
\bigl( u \otimes \nvec{f} \otimes \wh{f( t )} \bigr)$
is locally integrable for all $u \in \ini$ and $f \in \step$, and
\item[(v)] $t \mapsto \Delta_{\ini \otimes \fock} F_t %
\bigl( u \otimes \nvec{f} \otimes \wh{f( t )} \bigr)$
is locally square-integrable for all $u \in \ini$ and $f \in \step$.
\end{mylist}
The integrable process $F$ is \emph{bounded} if
$F_t \in \bop{\ini \otimes \fock \otimes \mmul}{}$ for
all~$t \in \R_+$.

Given such an integrable process $F$, the quantum stochastic integral
$( \int_0^t F_s \std \Lambda_s )_{t \in \R_+}$ is the unique operator
process such that
\[
\langle u \otimes \nvec{f}, \int_0^t F_s \std \Lambda_s %
v \otimes \nvec{g} \rangle = %
\int_0^t \langle u \otimes \nvec{f}, E^{\wh{f( s )}} F_s %
E_{\wh{g( s )}} v \otimes \nvec{g} \rangle \std s
\]
for all $t \in \R_+$, $u$, $v \in \ini$ and $f$, $g \in \step$. Such a
process is necessarily strongly continuous. For details see
Theorem~3.13 of \cite{Lin05}.
\end{definition}

\begin{remark}
\begin{mylist}
\item[(i)] Definition~\ref{dfn:integrable}(iii) implies that a
bounded integrable process~$F$ is strongly measurable everywhere on
$\ini \otimes \fock \otimes \mmul$, that is, the map
$t \mapsto F_t \theta$ is strongly measurable for all
$\theta \in \ini \otimes \fock \otimes \mmul$.
\item[(ii)] If $X$ is a strongly measurable bounded operator process
on $\ini$ which has locally bounded norm then
$\wt{X} := ( X_t \otimes I_\mmul )_{t \in \R_+}$ is a bounded
integrable process with locally bounded norm.
\end{mylist}
\end{remark}

\begin{lemma}\label{lem:timedptqsde}
Let $R \in \bop{\ini}{}$ and let $L$ be a bounded integrable process with locally bounded norm. There is a unique strong solution to the
QSDE
\begin{equation}\label{eqn:timedptQSDE}
X_0 = R \otimes I_\fock \qquad \text{and} \qquad %
\rd X_t = L_t \wt{X}_t \std \Lambda_t,
\end{equation}
where $\wt{X_t} := X_t \algten I_\mmul$. The solution $X$ is a
strongly continuous operator process and $\wt{X}$ is an
integrable process such that $t \mapsto X_t E_{\nvec{f}}$ has locally
bounded norm for every $f \in \step$.
\end{lemma}
\begin{proof}
This is essentially Proposition~3.1 of~\cite{GLW01}, rewritten in the
notation and language of~\cite{Lin05}. In particular $X$ is found by
solving the iteration scheme
\[
X^{(0)}_t = R \otimes I_\fock,
\qquad
X^{(n+1)}_t = R \otimes I_\fock +\int^t_0 L_s \wt{X}^{(n)}_s \std
\Lambda_s, \quad t \in \R_+, n \in \Z_+
\]
where $\wt{X}^{(n)}_t := X^{(n)}_t \algten I_\mmul$. The norm continuity of
$t \mapsto \wt{X}^{(n)}_t \xi$
for all $\xi \in \ini \algten \evecs( \tot ) \algten \mmul$
implies that~$s \mapsto L_s \wt{X}^{(n)}_s$ is an integrable process, giving
the existence of $X^{(n+1)}$. We fix $f \in \step$ and let
\[
Y^{(0)}_t := ( R \otimes I_\fock ) E_{\nvec{f}} \qquad \text{and} \qquad
Y^{(n)}_t := (X^{(n+1)}_t - X^{(n)}_t) E_{\nvec{f}}.
\]
Given $T > 0$ and applying~\cite[Theorem~3.13]{Lin05},
it follows by induction that
$Y^{(n)}_t \in \bop{\ini}{\ini \otimes \fock}$ for all $n \in \N$,
with
\[
\sup_{t \in [0,T]} \| Y^{(n)}_t u \|^2 \le K_{f,T} \int^t_0 %
\| L_s \|^2 \| \wh{f}( s ) \|^2 \| Y^{(n-1)}_s u \|^2 \std s,
\]
where $K_{f, T}$ is a constant that depends only on $f$ and $T$. Consequently,
we have that
\[
\sup_{t \in [0,T]} \| Y^{(n)}_t u \| \le \frac{1}{\sqrt{n!}} %
K^{n/2}_{f,T} M^n_T \| 1_{[0,T]} \wh{f} \|^n \| R \| %
\| \nvec{f} \| \| u \|,
\]
where $M_T = \sup_{t \in [0,T]} \| L_t \|$. Setting
$X_t u \otimes \nvec{f} = \sum_{n=0}^\infty Y^{(n)}_t u$ for all
$u \in \ini$ and $t \in \R_+$ gives a process~$X$. It is now a routine matter to check
that $X$ is strongly continuous, and integrability of $\wt{X}$ then
follows, in part because of the manifestly locally bounded norm of
$t \mapsto X_t E_{\nvec{f}}$. Thus~$X$ can be shown to
satisfy~(\ref{eqn:timedptQSDE}). Uniqueness follows by taking the
difference of two solutions and iterating as above.
\end{proof}

\begin{proposition}\label{prp:multiplier}
Let $j$ be a completely bounded mapping process on the operator space $\opsp$, with locally bounded completely bounded norm. If
$F \in \opsp \matten \bop{\mmul}{}$ is such that
$t \mapsto \wt{\jmath}_t( F )$ is strongly measurable then there is a
unique strong solution to the QSDE
\begin{equation}\label{eqn:multqsde}
X_0 = I_{\ini \otimes \fock} %
\qquad \text{and} \qquad %
\rd X_t = \wt{\jmath}_t( F ) \wt{X}_t \std \Lambda_t. %
\end{equation}
The solution $X$ is a strongly continuous process. If $\opsp = \alg$, a
unital $C^*$~algebra, with each $j_t$ being a $*$-homomorphism and
such that
\begin{equation}\label{eqn:jFDeltaF}
\wt{\jmath}_t( F )^* \Delta_{\ini \otimes \fock} \wt{\jmath}_t( F )%
= \wt{\jmath}_t( F^* \Delta_\ini F ),
\end{equation}
then $X$ is contractive if and only if
\[
q( F ) := F + F^* + F^* \Delta_\ini F \le 0
\]
and $X$ is isometric if and only if $q( F ) = 0$.
\end{proposition}
\begin{proof}
The existence of $X$ is an application of
Lemma~\ref{lem:timedptqsde}. The extra hypotheses on $\opsp$,
$j$ and~$F$, together with the weak form of the quantum It\^o
product formula \cite[Theorem~3.15]{Lin05}, imply that
\[
\| X_t \theta \|^2 - \| \theta \|^2 = %
\int_0^t \langle \wt{X}_s \wh{\nabla}_s \theta, %
\wt{\jmath}_s( F + F^* + F^* \Delta F ) \wt{X}_s %
\wh{\nabla}_s \theta \rangle \std s %
\quad \text{for all } \theta \in \ini \algten \evecs( \tot )
\text{ and } t \in \R_+,
\]
where
$\wh{\nabla}_s u \otimes \nvec{f} := u \otimes \nvec{f} \otimes \wh{f(  s )}$
and this definition is extended by linearity. Sufficiency of the isometry and contractivity conditions follows
immediately; for the latter, recall that matrix-space liftings
preserve completely positivity, by
Lemma~\ref{lem:cplifting}. Necessity follows by differentiating at
$0$: the integrand is continuous at $0$ by
Remark~\ref{rmk:procprops}(ii).
\end{proof}

\begin{remark}\label{rem:mult}
\begin{mylist}
\item[(i)] If $\opsp = \alg$, a unital $C^*$~algebra, and $j_t$ is a
$*$-homomorphism, then (\ref{eqn:jFDeltaF}) holds whenever
$F \in \alg \otimes \bop{\mmul}{}$.
\item[(ii)] Suppose that $I_\ini \in \opsp$ and
$j_t( I_\ini ) = I_{\ini \otimes \fock}$ for all $t \in \R_+$. If
$C \in \bop{\mmul}{}$ and $F = I_\ini \otimes C$ then the conditions of
Proposition~\ref{prp:multiplier} are satisfied and the strongly
continuous operator process $X$ such that
\[
X_t = I_{\ini \otimes \fock} + \int_0^t X_s \algten C \std \Lambda_s
\qquad \text{strongly on } \ini \algten \evecs(\tot) %
\qquad \text{for all } t \in \R_+
\]
is contractive, isometric or co-isometric if and only if
$q( F ) \le 0$, $q( F ) = 0$ or $q( F^* ) = 0$, respectively.
This follows from the quantum It\^o product formula,
Theorem~\ref{thm:qipf}; alternatively, note that this is the usual
Hudson--Parthasarathy QSDE with time-independent coefficients.
\end{mylist}
\end{remark}

\begin{lemma}\label{lem:measurability}
Let $j$ be a completely bounded mapping process on the operator space
$\opsp$, with locally bounded completely bounded norm.
\begin{mylist}
\item[\tu{(i)}] The process $t \mapsto \wt{\jmath}_t( F )$ is
weakly measurable for all $F \in \opsp \matten \bop{\mmul}{}$ if and
only if the process $t \mapsto j_t( x )$ is weakly measurable for all
$x \in \alg$.
\item[\tu{(ii)}] If $t \mapsto j_t( x )$ is strongly measurable
for all $x \in \opsp$ then $t \mapsto \wt{\jmath}_t( F )$ is strongly
measurable for all $F \in \alg \otimes \bop{\mmul}{}$,
the spatial tensor product.
\item[\tu{(iii)}] If $j$ is strongly continuous then
$t \mapsto \wt{\jmath}_t ( F )$ is strongly measurable for all
$F \in \opsp \matten \bop{\mmul}{}$, the matrix-space tensor product.
\end{mylist}
\end{lemma}
\begin{proof}
Parts~(i) and~(ii) are easily checked.

For part~(iii), it is enough to show that if
$F \in \opsp \matten \bop{\mmul}{}$,
$u \in \ini$, $f \in \step$ and $z \in \mmul$ then the map
$t \mapsto \wt{\jmath}_t (F) \zeta$ is strongly measurable, where
$\zeta := u \otimes \nvec{f} \otimes z$.

We fix an orthonormal basis
$\{ e_\alpha \}_{\alpha \in A}$ of $\mmul$ and let
$E^\alpha: = E^{e_\alpha}$. Given any
$\xi \in \ini \otimes \fock \otimes \mmul$, we have that
$\| \xi \|^2 = \sum_{\alpha \in A} \| E^\alpha \xi \|^2$ and so
$E^\alpha \xi \neq 0$ for only countably many~$\alpha$. Thus
there is a countable set $A_t \subseteq A$ for each $t \in \R_+$
such that
$E^\alpha \wt{\jmath}_t (F) \zeta = 0$ when $\alpha \notin A_t$.
We let $\hilc_\Q$ be the
separable closed subspace of $\mmul$ with orthonormal basis
$\{e_\alpha : \alpha \in A_\Q \}$, where
$A_\Q := \bigcup\{ A_t : t \in \Q \cap \R_+ \}$.

Next, for each $\alpha \in A_\Q$ we know that
$t \mapsto j_t ( E^\alpha F E_z ) u \otimes \nvec{f}$ is continuous, so
its image is contained in a separable subspace $\hilb_\alpha$
of~$\ini \otimes \fock$. Let $\hilb_\Q$ denote the smallest closed
subspace of~$\ini \otimes \fock$ that contains every $\hilb_\alpha$,
which will also be separable. Then for each $t \in \Q \cap \R_+$ we have that
\[
\wt{\jmath}_t ( F ) \zeta = \sum_{\alpha \in A_t} \bigl[ %
j_t ( E^\alpha F E_z ) u \otimes \nvec{f} \bigr] \otimes3 e_\alpha \in %
\hilb_\Q \otimes \hilc_\Q.
\]
Finally, for each $t \in \R_+ \setminus \Q$, we have that
\[
\langle \nu, \wt{\jmath}_t( F ) \zeta \rangle = %
\lim_{n \to \infty} \langle \nu, \wt{\jmath}_{t_n} ( F ) \zeta \rangle
\]
for any $\nu \in \ini \otimes \fock \otimes \mmul$ and any sequence
$( t_n )_{n \in \N}$ in $\Q \cap \R_+$ that converges to $t$, by the strong, and hence
weak, continuity of $j$. Thus $\wt{\jmath}_t ( F ) \zeta$ belongs to
the weak closure of $\hilb_\Q \otimes \hilc_\Q$, which coincides
with the norm closure (as these closures coincide for convex subsets
of any normed vector space).

This shows that the vector process $( \wt{\jmath}_t ( F ) \zeta)_{t \in \R_+}$
is separably valued. The result now follows from Pettis' Theorem, given that
this process is also weakly measurable as a result of the local boundedness
of the completely bounded norm of $j$ and its strong continuity.
\end{proof}

\begin{proposition}\label{prp:Fdecomp}
Let $\alg$ be a unital $C^*$ algebra with positive cone $\alg_+$ and suppose
\[
F = \begin{bmatrix} k & m \\[1ex]
 l & w - I_{\ini \otimes \mul} \end{bmatrix} \in %
\alg \matten \bop{\mmul}{}.
\]
Then $q( F ) \le 0$ if and only if $w \in \alg \matten \bop{\mul}{}$
is a contraction, $d:= -( k + k^* + l^* l ) \in \alg_+$ and there
exists a contraction $v$ such that
$m = -l^* w - d^{1 / 2} v ( I_{\ini \otimes \mul} - w^* w )^{1 / 2}$.

Furthermore, $q( F ) = 0$ if and only if
$w^* w = I_{\ini \otimes \mul}$, $k + k^* + l^* l = 0$ and
$m = -l^* w$.
\end{proposition}
\begin{proof}
The first part follows from standard characterisations of positive
$2 \times 2$ operator matrices (see~\cite[Lemma~2.1]{GLSW03}). The
second part is immediate.
\end{proof}

\subsection{Perturbation of the free flow}\label{sec:perturb}

\begin{remark}\label{rmk:notalg}
If $\alg \subseteq \bop{\ini}{}$ is a unital $C^*$ algebra and $\hilb$
is a Hilbert space then, in general, the operator space
$\alg \matten \bop{\hilb}{}$ is not an algebra, although it is always
an operator system in~$\bop{\ini \otimes \hilb}{}$: see
\cite[pp.~615--6]{LiW01}. The reason is already apparent at the level
of row and column spaces, since $\alg \otimes \ket{\hilb}$ and
$\alg \matten \ket{\hilb}$ typically differ, with the former having a
natural Hilbert $C^*$-bimodule structure not shared with the
latter. Moreover, if $\{e_\alpha\}_{\alpha \in A}$ is any orthonormal
basis of $\hilb$ then for each $T \in \alg \otimes \ket{\hilb}$ it is
the case that
\[
T = \sum_{\alpha \in A} E^{e_\alpha} T \otimes \ket{e_\alpha} =
\sum_{\alpha \in A} E_{e_\alpha} E^{e_\alpha} T
\]
with the series being norm convergent, whereas if
$T \in \alg \matten \ket{\hilb}$ then the series above converges to
$T$ in the strong operator topology, but not necessarily in norm.

The following construction provides a means of avoiding some of the
problems caused by this inconvenient difference.
\end{remark}

\begin{definition}
Let $\opsp \subseteq \bop{\ini}{}$ be an operator space, and let
$\hilb$ be a Hilbert space. Let
\begin{align*}
R( \opsp; \hilb ) & := %
\{ T \in \bop{\ini \otimes \hilb}{} : E^z T \in %
\opsp \otimes \bra{\hilb} \text{ for all } z \in \hilb \} \\[1ex]
\text{and} \quad %
C( \opsp; \hilb ) & := %
\{ T \in \bop{\ini \otimes \hilb}{} : T E_w \in %
\opsp \otimes \ket{\hilb} \text{ for all } w \in \hilb \}.
\end{align*}
\end{definition}

\begin{lemma}\label{lem:row+colspaces}
Let $T_R$ and $T_C$ denote the topologies on $\bop{\ini \otimes \hilb}{}$
generated by the families of seminorms $\{ \zp \}_{z \in \hilb}$
and $\{ \pw \}_{w \in \hilb}$, respectively, where
\[
\zp( T ) := \| E^z T \| %
\quad \text{and} \quad %
\pw( T ) := \| T E_w \| \qquad \text{for all } T \in \bop{\ini \otimes \hilb}{}.
\]
Then $R( \opsp; \hilb ) = \overline{\opsp \algten \bop{\hilb}{}}^R$
and $C( \opsp; \hilb ) = \overline{\opsp \algten \bop{\hilb}{}}^C$,
the closures with respect to $T_R$ and $T_C$, respectively. Moreover
\[
\opsp \otimes \bop{\hilb}{} \subseteq R( \opsp; \hilb ) \cap %
C( \opsp; \hilb ) \subseteq R( \opsp; \hilb ) \cup %
C( \opsp; \hilb ) \subseteq \opsp \matten \bop{\hilb}{}.
\]
\end{lemma}
\begin{proof}
Any $T \in \overline{\opsp \algten \bop{\hilb}{}}^R$ is the limit of a
$T_R$-convergent net $\{ T_i \}_{i \in I} \subseteq \opsp \algten \bop{\hilb}{}$.
We have that
$E^z T_i \in \opsp \algten \bra{\hilb} \subseteq \opsp \otimes \bra{\hilb}$,
with the latter space being norm closed. Thus $T \in R( \opsp; \hilb )$.

Next, let $S \in R( \opsp; \hilb )$. Given an orthonormal basis
$\{e_\alpha\}_{\alpha \in A}$ of $\hilb$, for each finite set
$B \finsubset A$ we let
\[
p_B := \sum_{\beta \in B} \dyad{e_\beta}{e_\beta} %
\quad \text{and} \quad %
P_B := I_\ini \otimes p_B = \sum_{\beta \in B} E_{e_\beta} E^{e_\beta}.
\]
For any finite set $C \finsubset A$, Remark~\ref{rmk:notalg} with $T = ( P_C S )^*$
gives that
\[
\lim_{B \finsubset A} \| P_C S P_B - P_C S \| = 0.
\]
We now consider the net
$\{ P_B S P_B \}_{B \finsubset A} \subseteq \opsp \algten \bop{\hilb}{}$.
Given any $z \in \hilb$ and $\varepsilon > 0$, there is a finite set
$B_0 \finsubset A$ such that $\| z - p_{B_0} z \| < \varepsilon$.
Then $P_B P_{B_0} = P_{B_0}$ for any finite set $B \supseteq B_0$ and so
\[
\zp( P_B S P_B - S ) \le 2 \varepsilon \| S \| + %
\| z \| \, \| P_{B_0} S P_B - P_{B_0} S \|.
\]
Hence $R( \opsp; \hilb ) \subseteq \overline{ \opsp \algten \bop{\hilb}{} }^R$,
as required. That $C( \opsp; \hilb ) = \overline{ \opsp \algten \bop{\hilb}{} }^C$
is proved in the same way. The inclusions are readily verified, in
particular as it is clear that $T_R$ and $T_C$ are weaker than the
norm topology.
\end{proof}

\begin{lemma}\label{lem:liftishom}
Let $\phi: \alg \to \blg$ be a $*$-homomorphism between the
$C^*$~algebras $\alg$ and $\blg$ and suppose  $S$, $T \in \alg \matten \bop{\hilb}{}$,
where $\hilb$ is a Hilbert space. If either $S \in R( \alg; \hilb )$ or $T \in C( \alg; \hilb )$
then $S T \in \alg \matten \bop{\hilb}{}$ and
\[
(\phi \matten \id_{\bop{\hilb}{}})( S ) %
(\phi \matten \id_{\bop{\hilb}{}})( T ) = %
(\phi \matten \id_{\bop{\hilb}{}})( S T ).
\]
\end{lemma}
\begin{proof}
For any orthonormal basis $\{ e_\alpha \}_{\alpha \in A}$ of $\hilb$
and any $z$, $w \in \hilb$ we have that
\[
E^z S T E_w = \sum_{\alpha \in A} E^z S E_{e_\alpha} E^{e_\alpha} T E_w \in \alg,
\]
as this series is norm convergent by Remark~\ref{rmk:notalg}. Hence
\begin{align*}
E^z ( \phi \matten \id_{\bop{\hilb}{}} ) ( S ) %
(\phi \matten \id_{\bop{\hilb}{}}) ( T ) E_w & = %
\sum_{\alpha \in A} E^z ( \phi \matten \id_{\bop{\hilb}{}} ) ( S ) E_{e_\alpha} %
E^{e_\alpha} ( \phi \matten \id_{\bop{\hilb}{}} ) ( T ) E_w \\[1ex]
 & = \sum_{\alpha \in A} \phi( E^z S E_{e_\alpha} ) \phi( E^{e_\alpha} T E_w ) \\[1ex]
 & = \phi\biggl( \sum_{\alpha \in A} E^z S E_{e_{\alpha}} E^{e_{\alpha}} T E_w \biggr) \\[1ex]
 & = E^z ( \phi \matten \id_{\bop{\hilb}{}} ) ( S T ) E_w,
\end{align*}
using the fact that $\phi$ is homomorphic and norm continuous.
\end{proof}

\begin{remark}
The ideas in the proof above also show that
$(\phi \matten \id_{\bop{\hilb}{}})( R( \alg; \hilb )) \subseteq R(
\blg; \hilb )$ and
$(\phi \matten \id_{\bop{\hilb}{}})( C( \alg; \hilb )) \subseteq C(
\blg; \hilb )$.
\end{remark}

\begin{lemma}\label{lem:prod}
Let $\Phi : \opsp \to \opsq$ be completely bounded, where
the operator space $\opsp \subseteq \bop{\ini_1}{}$ and
the operator space $\opsq \subseteq \bop{\ini_2}{}$.
If $R$, $S \in \bop{\hilb}{}$ for some Hilbert space $\hilb$ then
\[
( I_{\ini_2} \otimes R ) %
\bigl( ( \Phi \matten \id_{\bop{\hilb}{}} )( T ) \bigr) %
( I_{\ini_2} \otimes S ) = %
( \Phi \matten \id_{\bop{\hilb}{}} )\bigl( ( I_{\ini_1} \otimes R ) %
T ( I_{\ini_1} \otimes S ) \bigr)
\]
for all $T \in \opsp \matten \bop{\hilb}{}$.
\end{lemma}
\begin{proof}
This follows immediately from the definition.
\end{proof}

\begin{theorem}[Quantum It\^o Product Formula]\label{thm:qipf}
Let $X$ and $Y$ be bounded operator processes
and let $F^*$ and $G$ be bounded integrable processes
such that
\[
X^*_t = X^*_0 + \int_0^t F^*_s \std \Lambda_s \quad \text{and} \quad %
Y_t = Y_0 + \int_0^t G_s \std \Lambda_s \qquad \text{for all } %
t \in \R_+.
\]
If
$H = ( F_t \wt{Y}_t + \wt{X}_t G_t + F_t \Delta_{\ini \otimes \fock} G_t )_{t \in \R_+}$
is an integrable process then
\[
X_t Y_t = X_0 Y_0 + \int_0^t H_s \std \Lambda_s %
\qquad \text{for all } t \in \R_+.
\]
\end{theorem}
\begin{proof}
See \cite[Corollary~3.16]{Lin05}, to which we have added the
possibility of having non-zero initial values.
\end{proof}

\begin{remark}
If $s \mapsto H_s$ is only weakly rather than strongly measurable, but
with $s \mapsto \langle \xi, H_s \zeta \rangle$ locally integrable for
suitable choices of $\xi$ and $\zeta$, then the product process $XY$
will only possess a weak integral representation.
\end{remark}

\begin{theorem}\label{thm:perturbint}
Let $j$ be a quantum stochastic flow on the unital $C^*$ algebra
$\alg$, let $\tot$ be an admissible set and let
$\phi : \alg_0 \to \alg \matten \bop{\mmul}{}$ be a linear map such
that the QSDE~(\ref{eqn:flowqsde}) holds strongly on $\ini \algten
\evecs( \tot )$ for all $x \in \alg_0$.

Let $X$ and $Y$ be solutions to the multiplier
equation~(\ref{eqn:multqsde}) with generators~$F$ and~$G$, respectively,
and each with locally bounded norm. Suppose that $F^* \Delta T$,
$T \Delta G$, and~$F^* \Delta T \Delta G$ are elements of
$ \alg \matten \bop{\mmul}{}$ for all~$T \in \im \phi$, with
\begin{subequations}
\begin{align}
\wt{\jmath}_t( F^* \Delta ) \wt{\jmath}_t( T ) & = %
\wt{\jmath}_t( F^* \Delta T ), \label{eqn:cond1a} \\[1ex]
\wt{\jmath}_t( T ) \wt{\jmath}_t( \Delta G ) & = %
\wt{\jmath}_t( T \Delta G ) \label{eqn:cond1b} \\[1ex]
\text{ and } \qquad %
\wt{\jmath}_t( F^* \Delta ) \wt{\jmath}_t( T ) %
\wt{\jmath}_t( \Delta G ) & = %
\wt{\jmath}_t( F^* \Delta T \Delta G ) \qquad %
\text{for all } t \in \R_+. \label{eqn:cond1c}
\end{align}
\end{subequations}
Suppose also that
$F^* \Delta ( x \otimes I_\mmul ) \Delta G \in %
\alg \matten \bop{\mmul}{}$
for all $x \in \alg_0$, with
\begin{equation}\label{eqn:cond2}
\wt{\jmath}_t( F^* \Delta ) %
\wt{\jmath}_t\bigl( ( x \otimes I_\mmul ) \Delta G \bigr) = %
\wt{\jmath}_t( F^* \Delta ( x \otimes I_\mmul ) \Delta G ) %
\qquad \text{for all } t \in \R_+.
\end{equation}
The completely bounded mapping process~$k$ on $\alg$ defined by setting
\[
k_t : \alg \to \bop{\ini \otimes \fock}{}; \ %
x \mapsto X_t^* j_t( x ) Y_t \qquad \text{for all } t \in \R_+
\]
is such that
\begin{equation}\label{eqn:perturbedQSDE}
\rd k_t( x ) = ( \wt{k}_t \comp \psi )( x ) \std \Lambda_t %
\qquad \text{weakly on } \ini \algten \evecs( \tot ) %
\qquad \text{for all } x \in \alg_0,
\end{equation}
where
\begin{align}\label{eqn:psi}
\psi : \alg_0 & \to \alg \matten \bop{\mmul}{}; \nonumber \\[1ex]
x & \mapsto ( I_{\ini \otimes \mmul} + \Delta F )^* \phi( x ) %
( I_{\ini \otimes \mmul} + \Delta G ) + F^* ( x \otimes I_\mmul ) + %
F^* \Delta ( x \otimes I_\mmul ) \Delta G + ( x \otimes I_\mmul ) G.
\end{align}
If $\ini$ and $\mul$ are separable then $k$
satisfies~(\ref{eqn:perturbedQSDE}) strongly.
\end{theorem}
\begin{proof}
Let $x \in \alg_0$. We may apply Theorem~\ref{thm:qipf} to the
processes $t \mapsto j_t( x )$ and $Y$, noting that the process
$t \mapsto j_t( x )^* = j_t( x^* )$
has a stochastic integral representation since $x^* \in \alg_0$. Thus
\[
j_t( x ) Y_t = x \otimes I_\fock +\int^t_0 H_s \std \Lambda_s,
\]
where
\begin{align*}
H_s & = \wt{\jmath}_s \bigl( \phi( x ) \bigr) \wt{Y}_s %
+ \bigl( j_s( x ) \otimes I_\mmul \bigr) \wt{\jmath}_s( G ) \wt{Y}_s %
+ \wt{\jmath}_s \bigl( \phi( x ) \bigr) \Delta \wt{\jmath}_s( G ) %
\wt{Y}_s \\[1ex]
 & = \wt{\jmath}_s \bigl( \phi( x ) +(x \otimes I_\mmul) G %
+\phi( x ) \Delta G \bigr) \wt{Y}_s,
\end{align*}      
where we have used~(\ref{eqn:cond1b}), Lemma~\ref{lem:liftishom} and
Lemma~\ref{lem:prod} to combine the three terms. This is valid
provided $H$ is an integrable process. To see this, note that
$s \mapsto Y_s$ is continuous in the strong operator topology, with
locally bounded norm, and so the same is true for $s \mapsto \wt{Y}_s$. Also,
the map
$s \mapsto \wt{\jmath} \bigl( A \bigr) \xi$ is
strongly measurable by Lemma~\ref{lem:measurability}, again with
bounded norm, for any~$A \in \alg \matten \bop{\mmul}{}$ and
$\xi \in \ini \otimes \fock \otimes \mmul$.

We next apply Theorem~\ref{thm:qipf} once again, to the processes $X^*_t$
and $j_t( x ) Y_t$, so that
\[
k_t( x ) = X^*_t j_t( x ) Y_t = x \otimes I_\fock
+\int^t_0 L_s \std \Lambda_s
\qquad \text{weakly},
\]
where
\[
L_s := \wt{X}^*_s \wt{\jmath}_s( F^* ) \bigl( j_s( x ) Y_s \otimes %
I_\mmul ) + \wt{X}^*_s \wt{\jmath}_s \bigl( \theta( x ) \bigr) %
\wt{Y}_s +\wt{X}^*_s \wt{\jmath}_s( F^* ) \Delta \wt{\jmath}_s %
\bigl( \theta( x ) \bigr) \wt{Y}_s
\]
and
\[
\theta( x ) := \phi( x ) +( x \otimes I_\mmul ) G +\phi( x ) \Delta G.
\]
Lemmas~\ref{lem:liftishom} and~\ref{lem:prod}, together with
assumptions~(\ref{eqn:cond1a}), (\ref{eqn:cond2})
and~\eqref{eqn:cond1c}, now show that
\[
L_s = \wt{k}_s \bigl( \psi( x ) \bigr),
\]
as required. Although the process $s \mapsto X_s$ is strong operator continuous, this is
not guaranteed for $s \mapsto X^*_s$.
\end{proof}

\begin{remark}
Identity~(\ref{eqn:cond2}) appears to treat $F$ and $G$ in an
asymmetrical fashion. However, Lemma~\ref{lem:liftishom} implies that
\[
\wt{\jmath}_t( F^* \Delta ) %
\wt{\jmath}_t \bigl( ( x \otimes I_\mmul ) \Delta G \bigr) %
= \wt{\jmath}_t \bigl( F^* \Delta ( x \otimes I_\mmul ) ) %
\wt{\jmath}_t ( \Delta G ).
\]
Using the right-hand side, one can carry out the proof above by first
looking at $X^*_t j_t( x )$ and then multiplying on the right by
$Y_t$.
\end{remark}

\begin{remark}
If $\alg$ is a von~Neumann algebra and each $j_t$ is ultraweakly
continuous then identities (\ref{eqn:cond1a}--c) and~(\ref{eqn:cond2})
hold without any need for assumptions on $F$ and $G$; in this case, the
matrix-space lifting $j_t \matten \id_{\bop{\mmul}{}}$ is the same as
the ultraweak tensor product $j_t \uwkten \id_{\bop{\mmul}{}}$.
\end{remark}

\begin{remark}\label{rem:col}
For identities (\ref{eqn:cond1a}--c) and~(\ref{eqn:cond2}) to hold, it
is sufficient that $\Delta F$ and~$\Delta G$ lie
in~$C( \alg; \mmul )$, by Lemma~\ref{lem:liftishom}. If $\mul$ is
finite dimensional then this condition is automatic.
\end{remark}

\begin{remark}\label{rem:apply}
\begin{mylist}
\item[(i)] Since
\[
\alg \to \bop{\ini \otimes \fock}{}; \ x \mapsto %
F^* ( x \otimes I_\mmul ) + %
F^* \Delta ( x \otimes I_\mmul ) \Delta G + ( x \otimes I_\mmul ) G
\]
is a bounded map for all $F$, $G \in \alg \matten \bop{\mmul}{}$, in
order to apply Theorem~\ref{thm:intcocycle} to the process~$k$
produced by Theorem~\ref{thm:perturbint}, it suffices to find, for all
$z, w \in \tot$, a norm-dense subspace $\alg_{z, w} \subseteq \alg_0$
such that the map
\[
\alg_{z, w} \to \alg; \ x \mapsto %
E^{\wh{z}} ( I_{\ini \otimes \mmul} + \Delta F )^* \phi( x ) %
( I_{\ini \otimes \mmul} + \Delta G ) E_{\wh{w}}
\]
is closable and its closure generates a $C_0$~semigroup.
\item[(ii)] Suppose further that
$\phi : \alg_0 \to \alg \matten \bop{\mmul}{}$ has standard form, as
in Definition~\ref{dfn:standard}, and consider the block-matrix
decompositions
\begin{equation}\label{eqn:block}
F = \begin{bmatrix} k_F & m_F \\[1ex]
 l_F & w_F - I_{\ini \otimes \mul} \end{bmatrix} %
\qquad \text{and} \qquad %
G = \begin{bmatrix} k_G & m_G \\[1ex]
 l_G & w_G - I_{\ini \otimes \mul} \end{bmatrix}.
\end{equation}
A short calculation shows that, for any $x \in \alg_0$,
\begin{multline}\label{eqn:genblock}
( I_{\ini \otimes \mmul} + \Delta F )^* \phi( x ) %
( I_{\ini \otimes \mmul} + \Delta G ) \\[1ex]
= \begin{bmatrix} \tau( x ) + l_F^* \delta( x ) + %
\delta^\dagger( x ) l_G & \delta^\dagger( x ) w_G \\[1ex]
w^*_F \delta( x ) & 0 \end{bmatrix} + %
\begin{bmatrix} l_F^* \\[1ex] w_F^* \end{bmatrix} %
\bigl( \pi( x ) - x \otimes I_\mul ) %
\begin{bmatrix} l_G & w_G \end{bmatrix}
\end{multline}
and, with $\psi$ as defined in (\ref{eqn:psi}), the quantity $\psi( x )$
equals
\begin{equation}\label{eqn:psidetailed}
\begin{bmatrix}
 \tinymaths{\tau( x ) + %
l_F^* \delta( x ) + l_F^* \pi( x ) l_G + \delta^\dagger( x ) l_G + %
k_F^* x + x k_G} & %
\tinymaths{\delta^\dagger( x ) w_G + l_F^* \pi( x ) w_G + x m_G} %
\\[1ex]
\tinymaths{w^*_F \delta( x ) + w^*_F \pi( x ) l_G + m_F^* x}
 & \tinymaths{w_F^* \pi( x ) w_G - x \otimes I_\mul}
\end{bmatrix}.
\end{equation}
If $\pi$ extends to a bounded operator from $\alg$ to
$\bop{\ini \otimes \mul}{}$ then, as a function of~$x$, the second
term in (\ref{eqn:genblock}) defines a bounded operator, so to apply
Theorem~\ref{thm:intcocycle} it suffices to find, for all $z$,
$w \in \tot$, a norm-dense subspace $\alg_{z, w}$ of~$\alg_0$ such
that the map
\[
\alg_{z, w} \to \alg; \ x \mapsto \tau( x ) + %
( l_F + w_F E_z )^* \delta( x ) + %
\delta^\dagger( x ) ( l_G + w_G E_w )
\]
is closable with closure that generates a $C_0$~semigroup.
\end{mylist}
\end{remark}

\begin{remark}
If $F = G$ then the map
\[
k_t : \alg \to \bop{\ini \otimes \fock}{}; \ %
a \mapsto X_t^* j_t( a ) X_t
\]
produced by Theorem~\ref{thm:perturbint} is completely positive, for
all~$t \in \R_+$. This map is unital if and only if $X$ is
isometric, and is a $*$-homomorphism if $X$ is co-isometric.

Moreover, if $q( F^* )= 0 = q( F )$, it is not technically difficult
to give a direct algebraic proof that the perturbed generator
$\psi = \left[ \begin{smallmatrix} \tau' & ( \delta' )^\dagger \\[0.5ex]
\delta' & \pi' - \iota \end{smallmatrix} \right]$ defined
by~(\ref{eqn:psidetailed}) is a generator in standard form. The
equality $q( F^* ) = 0$ shows that $\pi'$ is a $*$-homomorphism, that
$\delta'$ is a $\pi'$-derivation and that~$\tau'$ satisfies the
appropriate cohomological identity; $q( F ) = 0$ is only then needed
to show $w_F$ is isometric and hence $\psi( 1_\alg ) = 0$.
\end{remark}

\begin{example}
Continuing the previous Remark, a particular class of generators of interest are those that are
``gauge free'', that is, the zero map appears in the bottom right
corner of the $2 \times 2$ block-matrix decomposition. For the
generator $\phi$ of the free flow, this means that $\pi = \iota$, that is,
$\pi( x ) = x \otimes I_\mul$ for all $x \in \alg_0$. For the perturbed
generator $\psi$ to be gauge free as well thus requires that (dropping the
subscript on $w$)
\begin{equation}\label{eqn:gaugefree}
w^* ( x \otimes I_\mul ) w = x \otimes I_\mul \qquad \text{for each } x \in \alg_0.
\end{equation}
Setting $x = 1_\alg$ shows that $w$ must be isometric, which is one of
the conditions that must be satisfied to have $q( F ) = 0$, a
necessary condition for $X$ to be isometric. Note
that~(\ref{eqn:gaugefree}) holds whenever $w$ is an isometric element
of
$(\alg \otimes I_\mul)' \cap \bigl( \alg \matten \bop{\mul}{} \bigr) = %
Z( \alg ) \matten \bop{\mul}{}$,
where $Z( \alg )$ is the centre of $\alg$. If $w$ is co-isometric as
well, then belonging to $Z( \alg ) \matten \bop{\mul}{}$ is a
necessary and sufficient condition on $w$ to ensure that $\psi$ is
gauge free.
\end{example}

\section{Examples}\label{sec:egs}

\subsection{Weyl perturbations}

Let $j$ be a quantum stochastic flow on the $C^*$~algebra $\alg$ such
that
\begin{equation}\label{eqn:qsde}
j_0( x ) = x \otimes I_\fock \quad \text{and} \quad %
\rd j_t( x ) = \wt{\jmath}_t\bigl( \phi( x ) \bigr) \std \Lambda_t %
\qquad \text{strongly on } \ini \algten \evecs( \tot ) \qquad %
\end{equation}
$\text{for all } x \in \alg_0$, where $\alg_0$ is a norm-dense
$*$-subalgebra of $\alg$.

Let $F = I_\ini \otimes C$ for
\[
C = \begin{bmatrix}
\I h - \hlf \| c \|^2 & \ -\bra{U^* c} \\[1ex]
\ket{c} & U - I_\mul
\end{bmatrix} \in \bop{\mmul}{},
\]
where $h \in \R$, $c \in \mul$ and $U \in \bop{\mul}{}$ is unitary.
By Remark~\ref{rem:mult}(iv) and Proposition~\ref{prp:multiplier}
with multiplier generator $F = I_\ini \otimes C$, there exists a
unitary operator process $X$ such that
\[
X_t = I_{\ini \otimes \fock} + \int_0^t X_s \otimes C \std \Lambda_s %
\qquad \text{for all } t \in \R_+.
\]
Since
\[
\Delta F E_{\wh{z}} = %
I_\ini \otimes \ket{\xi} \in \alg \algten \ket{\mmul} %
\qquad \text{for all } z \in \mul,
\]
where $\xi = ( 0, c + U z - z ) \in \mmul$, it follows that
$\Delta F \in C\bigl( \alg; \mmul \bigr)$. Thus Remark~\ref{rem:col} and
Theorem~\ref{thm:perturbint} give a strongly continuous
$*$-homomorphic mapping process $k$.

If the free-flow generator $\phi : \alg_0 \to \alg \matten \bopp\bigl( \mmul \bigr)$
has standard form, so that
\smash[b]{$\phi = \left[ \begin{smallmatrix} \tau & \ \delta^\dagger \\[0.5ex]
\delta & \ \pi - \iota \end{smallmatrix} \right]$,} then the perturbed generator
$\psi$ of $k$ is such that $\psi( x )$ equals
\[
\begin{bmatrix}
\tau( x ) + E^c \delta( x ) + E^c \pi( x ) E_c + %
\delta^\dagger( x ) E_c - \| c \|^2 x & %
\delta^\dagger( x ) ( I_\ini \otimes U ) + %
E^c \pi( x ) ( I_\ini \otimes U ) - x E^{U^* c} \\[1ex]
( I_\ini \otimes U )^* \delta( x ) + %
( I_\ini \otimes U )^* \pi( x ) E_c - E_{U^* c} x & %
( I_\ini \otimes U )^* \pi( x ) ( I_\ini \otimes U ) - %
x \otimes I_\mul
\end{bmatrix}
\]
for all $x \in \alg_0$. In particular, if $\pi = \iota$, so that
$\phi = \left[ \begin{smallmatrix} \tau & \ \delta^\dagger \\[0.5ex]
\delta & \ 0 \end{smallmatrix} \right]$,
then
$\psi = \left[ \begin{smallmatrix} \tau' & \ ( \delta' )^\dagger \\[0.5ex]
\delta' & \ 0 \end{smallmatrix} \right]$, where
\[
\tau' : \alg_0 \to \alg; \ %
x \mapsto \tau( x ) + E^c \delta( x ) + \delta^\dagger( x ) E_c %
\quad \text{and} \quad %
\delta' : \alg_0 \to \alg \matten \ket{\mul}; \ %
x \mapsto ( I_\ini \otimes U )^* \delta( x ).
\]
In particular, we see that
\[
\psi^{\wh{z}}_{\wh{w}} = \phi^{\wh{c + U z}}_{\wh{c + U w}}
\qquad \text{for all } z, w \in \mul.
\]
If $j$ has a
strongly continuous vacuum-expectation semigroup then
Proposition~\ref{prp:assocsemigen} applies and shows that each
$\phi^{\wh{z}}_{\wh{w}}$ is closable and has an extension that
generates a $C_0$~semigroup. If in fact it is the closure of
$\phi^{\wh{z}}_{\wh{w}}$ that is the semigroup generator, and if
$c+Uz \in \tot$ for all $z \in \tot$, then $k$ is a Markovian cocycle
by Theorem~\ref{thm:intcocycle}. These conditions are satisfied by the
cocycles constructed in Theorems~3.9 and~3.12 of~\cite{BeW14} where
$\tot = \mul$, as shown by a combination of Theorem~3.16 and
Lemma~2.14 of that paper. These results show that $\alg_0$ is a core
for the generators of all of the associated semigroups.

\subsection{The quantum exclusion process}

Let $\ini = \fock_-\bigl( \ell^2( I ) \bigr)$, the Fermionic Fock
space over $\ell^2( I )$, where $I$ is a non-empty set, and let $b_i$
and $b_i^*$ be the annihilation and creation operators at site
$i \in I$, respectively, so that
\[
\{ b_i, b_j \} = 0 \qquad \text{and} %
\qquad \{ b_i, b_j^* \} = \tfn{i = j} %
\qquad \text{for all } i, j \in I.
\]
The \emph{CAR algebra} $\alg$ is the norm closure of $\alg_0$, the
$*$-algebra generated by $\{ b_i : i \in I \}$, in $\bop{\ini}{}$.

The quantum exclusion process was introduced by
Rebolledo~\cite{Reb05}, and the associated process constructed
in~\cite{BeW14} under certain assumptions discussed below. The two
inputs, given the choice of $I$, are functions $\eta: I \to \R$ and
$\alpha: I \times I \to \C$. The former gives the energy $\eta_i$ at each
site~$i \in I$, and $\alpha_{i, j}$ is an amplitude from site $i$ to site~$j$. We set
\[
t^\alpha_{i,j} := \alpha_{i, j} b^*_j b_i  \qquad \text{and} \qquad
\delta^\alpha_{i, j} : \alg \to \alg; \ x \mapsto [ t^\alpha_{i, j}, x ] \qquad
 ( i, j \in I ).
\]
Let $\mul$ have orthonormal basis $\{ f_{i, j}: i, j \in I \}$, so that
$\mul \cong l^2( I \times I )$. We can assemble the set of operators
$\{ t^\alpha_{i, j} : i, j \in I \}$
into a column operator and the set of derivations $\{ \delta^\alpha_{i, j} : i, j \in I \}$ into an
associated $\iota$-derivation:
\begin{align}
&t_\alpha := \sum_{i, j \in I} t^\alpha_{i, j} \otimes \ket{f_{i, j}} \label{eqn:talpha} \\[1ex]
\text{and} \quad & \delta_\alpha: \alg_0 \to \alg \matten \bop{\mul}{}; \ %
x \mapsto \sum_{i, j \in I} \delta^\alpha_{i, j}( x ) \otimes \ket{f_{i, j}} = %
t_\alpha x - ( x \otimes I_\mul ) t_\alpha. \label{eqn:delta alpha}
\end{align}
If the series in~(\ref{eqn:talpha}) converges to give a bounded
operator from $\ini$ to $\ini \otimes \mul$ then
$t_\alpha \in \alg \matten \ket{\mul}$; moreover, $\delta_\alpha$ is a well defined
$\iota$-derivation, and the domain of $\delta_\alpha$ can be extended
to all of $\alg$. However, the standing assumption
in~\cite{BeW14} does not necessarily give such convergence. Instead,
the following \emph{finite-valence} and \emph{symmetry} assumptions
were made:
\begin{equation}\label{eqn:finitevalence}
\{ j \in I: \alpha_{i, j} \neq 0 \} \text{ is finite for all } i \in I
\qquad \text{and} \qquad
| \alpha_{i, j} | = | \alpha_{j, i} | \text{ for all } i, j \in I.
\end{equation}
Under these assumptions it turns out that $\delta_\alpha( x )$ is
well defined for each $x \in \alg_0$ since only finitely many of the
terms in the series~(\ref{eqn:delta alpha}) are non-zero, even if
$t_\alpha$ is not a bounded operator.

The generator of the semingroup is then given by
\begin{align*}
\tau_{\alpha, \eta}: \alg_0 \to \alg_0; \ x & \mapsto %
\I \sum_{i \in I} \eta_i [ b^*_i b_i, x ] - %
\frac{1}{2} \sum_{i, j \in I \times I} \bigl( %
( t^\alpha_{i, j} )^* [ t^\alpha_{i, j}, x ] + %
[ x, ( t^\alpha_{i, j} )^*] t^\alpha_{i, j} \bigr) \\[1ex]
 & = \I [ H, x]  - \hlf t^*_\alpha \delta_\alpha( x ) - \hlf \delta^\dagger_\alpha( x ) t_\alpha,
\end{align*}
where $H := \sum_{i \in I} \eta_i b^*_i b_i$. Again, the series for
$H$ may not converge, but the commutator $[H, x]$ is well defined, as
are the products $t^*_\alpha \delta_\alpha( x )$ and
$\delta^\dagger_\alpha( x ) t_\alpha$, courtesy
of~(\ref{eqn:finitevalence}).

We can now write
\[
\phi = \begin{bmatrix} \tau_{\alpha, \eta} & \delta^\dagger_\alpha \\[1ex]
\delta_\alpha & 0 \end{bmatrix}: \alg_0 \to \alg_0 \algten \bop{\mmul}{}.
\]
This map has standard form, and under a variety of
hypotheses~\cite[Examples~5.11--13]{BeW14} it was shown that there
exists a quantum stochastic flow $j$ with $\phi$ as its generator.

For the remainder of this section we will assume that such a flow $j$
exists, and show how the methods of this paper can allow us to go
beyond the assumptions~(\ref{eqn:finitevalence}).

Let $\beta: I \times I \to \C$, and suppose that $t_\beta$ defined
through~(\ref{eqn:talpha}) is a well defined element of $\alg \matten \ket{\mul}$
such that $t^*_\beta t_\beta \in \alg$. Choose $h \in \alg$ such that $h = h^*$ and set
\[
F := \begin{bmatrix}
 \I h - \hlf t^*_\beta t_\beta & t^*_\beta \\[1ex]
 -t_\beta & 0 \end{bmatrix}.
\]
The strongly continuous operator process $X = X^F$ given by
Proposition~\ref{prp:multiplier} is isometric. Furthermore, since
$\Delta F = \left[\begin{smallmatrix} 0 & 0 \\[0.5ex]
 -t_\beta & 0 \end{smallmatrix} \right] \in C(\alg; \mmul)$,
Remark~\ref{rem:col} makes it clear that Theorem~\ref{thm:perturbint}
may be applied, when $I$ is countable, to obtain a completely positive
unital mapping process $k$ with generator
\begin{equation}\label{eqn:newQEPgen}
\psi = \begin{bmatrix} \tau' & (\delta')^\dagger \\ \delta' & 0 \end{bmatrix},
\end{equation}
where
\begin{align*}
\tau'( x  ) & = \tau_{\alpha, \eta}( x ) - t^*_\beta \delta_\alpha( x ) %
+ t^*_\beta (x \otimes I_\mul) t_\beta - \delta^\dagger_\alpha( x ) t_\beta - %
\hlf \{ t^*_\beta t_\beta, x \} - \I [ h, x ] \\[1ex]
\text{and} \quad \delta'( x ) & = %
\delta_\alpha( x ) - ( x \otimes I_\mul ) t_\beta + t_\beta x.
\end{align*}
for all $x \in \alg_0$. In particular for any choice of $\beta$ we have
$\delta' = \delta_{\alpha + \beta}$. Furthermore, if the series
\begin{equation}\label{eqn:difficultproduct}
t^*_\alpha \delta_\beta( x ) = %
\sum_{i, j \in I} \overline{\alpha_{i, j}} \beta_{i, j} b^*_i b_j [ x, b^*_j b_i ]
\end{equation}
is convergent in the weak operator topology for all $x \in \alg_0$ then the expression
defining $\tau'$ above can also be rigorously manipulated to show that
$\tau' = \tau_{\alpha +\beta, \eta}$. Thus $k$ is a process with a
generator of the same structure as $j$, but the amplitudes have been
changed. This can have two obvious benefits:
\begin{mylist}
\item[(i)] the symmetry condition
$| \alpha_{i, j} +\beta_{i, j} | = | \alpha_{j, i} +\beta_{j, i} |$
need not apply;
\item[(ii)] the finite-valence condition need not apply to $\alpha +\beta$.
\end{mylist}
The cost, currently, for circumventing these restrictions imposed in
our earlier work~\cite{BeW14} is that it is not yet known if the
resulting process $k$ is multiplicative.

As an example of conditions that ensure that the various series above
behave as required, we give one set of sufficient hypotheses.

\begin{theorem}
Let $\eta \in l^\infty( I )$, $\alpha \in l^\infty( I \times I )$ and $\beta \in l^1( I \times I )$, 
where $I$ is a countable set and $\alpha$ satisfies~(\ref{eqn:finitevalence}).
The process $k$ with generator~(\ref{eqn:newQEPgen}) exists and satisfies the
QSDE~(\ref{eqn:flowqsde}) strongly on $\ini \algten \evecs( \mul )$
for all $x \in \alg_0$.
\end{theorem}
\begin{proof}
Existence of $j$ is dealt with in Example~5.11 of~\cite{BeW14}. Since
$\alpha$ is bounded and $\beta$ is summable, the
series~(\ref{eqn:difficultproduct}) is norm convergent and so the
results of Sections~\ref{sec:multeqn} and~\ref{sec:perturb} apply.
\end{proof}

\begin{remark}
Only very minimal assumptions have been made regarding the graph
with $I$ as the set of vertices and an edge
between $i$ and $j$ whenever $\alpha_{i, j} +\beta_{i, j} \neq 0$. A
more detailed study of this graph would undoubtedly allow
less restrictive assumptions to be imposed on the perturbation
function $\beta$.
\end{remark}

\begin{remark}
An alternative approach to constructing quantum exclusion processes
has been developed in~\cite{LiW21}, based on an analysis of the
associated semigroups. No assumption of symmetry is made on the
amplitudes, but it is assumed that $I = \Z^d$ and that
$\alpha_{i, j} \neq 0$ only for sites $i$ and $j$ within a fixed distance of
each other. As in this paper, it is not yet known if the resulting
cocycle in~\cite{LiW21} is multiplicative.
\end{remark}

\subsection{Flows on universal $C^*$ algebras}

\begin{definition}\label{defn:gendalg}
Let $\alg$ be a unital $C^*$~algebra $\alg$ and let $\{ a_i : i \in I \}$ be a subset of $\alg$.
We let~$\words$ denote the set of all words in
the elements $\{ a_i, a_i^*: i \in I \}$, so that $\alg_0 = \lin \words$ is
the $*$-subalgebra generated by this subset. The set
$\words$ is said to \emph{generate} $\alg$ if $\alg_0$ is norm dense in~$\alg$.

These generators \emph{satisfy the relations} $\{ p_k : k \in K \}$ if
each $p_k$ is a complex polynomial in the non-commuting indeterminates
$\langle X_i, X_i^* : i \in I \rangle$ and the
algebra element $p_k( a_i , a_i^* : i \in I )$ obtained from $p_k$ by
replacing each $X_i$ by $a_i$ and~$X_i^*$ by $a_i^*$ is
equal to~$0$ for all $k \in K$. 

A generator $a_i$ is called \emph{balanced} if in each relation $p_k$
the difference between the number of instances of $a_i$ and the number
of instances of $a^*_i$ in each monomial making up $p_k$ is constant.

Let the unital $C^*$~algebra $\alg$ have generators $\{ a_i : i \in I \}$ that satisfy
the relations $\{ p_k : k \in K \}$. Then $\alg$ is \emph{universal} if, given any unital
$C^*$~algebra~$\blg$ containing a set of elements
$\{ b_i : i \in I \}$ which satisfies the relations
$\{ p_k : k \in K \}$, that is,~$p_k( b_i, b_i^* : i \in I ) = 0$ for
all $k \in K$, there exists a unique $*$-homomorphism
$\pi : \alg \to \blg$ such that~$\pi( a_i ) = b_i$ for all $i \in I$.
\end{definition}

\begin{example}
In~\cite{BeW14} we considered flows on the universal rotation algebra,
which is the universal algebra generated by the unitaries $U$, $V$ and
$Z$ such that
\[
U V = Z V U, \quad U Z = Z U \quad \text{and} \quad V Z = Z V.
\]
If we label the generators as $a_1 = U$, $a_2 = V$ and $a_3 = Z$ then
the full list of relations are
\begin{gather*}
p_1 = X_1 X^*_1 - 1, \quad p_2 = X_1^* X_1 - 1, \quad p_3 = X_2 X^*_2 - 1, %
\quad p_4 = X_2^* X_2 - 1, \quad p_5 = X_3 X^*_3 - 1, \\[1ex]
p_6 = X_3^* X_3 - 1, \quad p_7 = X_1 X_2 -X_3 X_2 X_1, %
\quad p_8 = X_1 X_3 - X_3 X_1, \quad p_9 = X_2 X_3 - X_3 X_2.
\end{gather*}
The relations $p_1$, \ldots, $p_6$ encode the unitarity of the
generators. Note that $U$ and $V$ are both balanced, but $Z$ is not
courtesy of relation $p_7$.
\end{example}

\begin{lemma}\label{lem:derivationsource}
Let $\alg$ be a universal $C^*$~algebra generated by
$\{ a_i : i \in I \}$ and let $\words$ and $\alg_0$ be as in
\tu{Definition~\ref{defn:gendalg}}. Suppose that the generating
element $a_j$ is balanced. There is a $C_0$ group of
automorphisms $\alpha^j$ with generator $\I d_j$ such that
\begin{mylist}
\item[\tu{(i)}] the domain $\dom d_j$ is a $*$-subalgebra containing $\alg_0$,
\item[\tu{(ii)}] the derivation $d_j$ is skew symmetric, that is, $d_j^\dagger = -d_j$, and
\item[\tu{(iii)}] we have $d_j ( b ) = n_j( b ) b$ for each $b \in \words$, where
\[
n_j(b) := \text{number of copies of $a_j$ in $b$} - %
\text{number of copies of $a^*_j$ in $b$}.
\]
\end{mylist}
\end{lemma}
\begin{proof}
For each $t \in \R$, let
\begin{equation}\label{eqn:newgenset}
b_i := \begin{cases} a_i & \text{if } i \neq j, \\ e^{\I t} a_j &
\text{if } i = j. \end{cases}
\end{equation}
Since $a_j$ is balanced it follows that $\{ b_i : i \in I \}$ is a set
of elements in $\alg$ that satisfy the same relations as the original
set of generators, and moreover this new set also generates
$\alg$. Universality implies that is an automorphism $\alpha^j_t$ of
$\alg$ such that $\alpha^j_t ( a_i ) = b_i$ for all~$i$. It is easy to
see that $\alpha^j_t ( b ) = e^{\I t n_j( b )} b$ for all
$b \in \words$ and $t \in \R$. Since $\| \alpha^j_t \| = 1$ for all
$t$, it follows by linearity and continuity that
$( \alpha^j_t )_{t \in \R}$ is a $C_0$ group of automorphisms.

That $\dom d_j$ is a $*$-subalgebra of $\alg$ and $d_j$ is skew symmetric
follows from the fact that $\alpha^j$ is a group of automorphisms. Moreover,
it is immediate that $\words \subseteq \dom d_j$, with $d_j( b ) = n_j( b ) b$ for
each $b \in \words$, and so $\alg_0 \subseteq \dom d_j$.
\end{proof}

\begin{theorem}\label{thm:genfrombalance}
Let $\alg$ be a universal $C^*$~algebra generated by the set
$\{ a_i : i \in I \}$, let $\{ a_j : j \in J \}$ be a subset of balanced generators, where $J \subseteq I$,
and let
$\mul := \ell^2( J )$ with the standard orthonormal basis $\{ e_j : j \in J \}$.
For any choices of constants $\{ c_j \in \C : j \in J \}$, the following defines a generator
$\phi: \alg_0 \to \alg_0 \otimes \bop{\mul}{}$ in standard form
according to~(\ref{eqn:blockgenform})\tu{:}
\begin{equation}\label{eqn:delta+tau}
\delta( x ) = \sum_{j \in J} c_j d_j( x ) \otimes \ket{ e_j }, %
\qquad \tau = -\hlf \sum_{j \in J} | c_j |^2 d_j^2 %
\quad \text{and} \quad \pi = \iota,
\end{equation}
where $d_j$ is as defined in~\tu{Lemma~\ref{lem:derivationsource}}. There is a weakly
multiplicative strong solution $j$ to~(\ref{eqn:flowqsde}) for this
generator $\phi$. If, in addition, the generators are all isometries
then $j$ is $*$-homomorphic.
\end{theorem}
\begin{proof}
The series for $\delta( x )$ and $\tau( x )$ both
have only finitely many non-zero terms, so are well defined, since
$d_j( a_k ) = \tfn{j = k} \I a_k$. Hence $\phi$ is a generator is
standard form, by Example~\ref{ex: gauge free skew derivs}.

Furthermore, if $j \in J$ then $\phi( a_j ) = a_j \otimes T_j$ for
\[
T_j := -\hlf | c_j |^2 \ket{f_0} \bra{f_0} + %
c_j \ket{f_j} \bra{f_0} - \overline{c_j} \ket{f_0} \bra{f_j} \in %
\bop{\mmul}{},
\]
where $f_0 = ( 1, 0 )$ and $f_j = ( 0, e_j ) \in \mmul$. On the other
hand, if $i \in I \setminus J$ then $\phi( a_i ) = 0$. It follows from
\cite[Corollary~2.12, Theorem~3.5 and Lemma~3.7]{BeW14} that we can
solve~(\ref{eqn:flowqsde}) to find a weakly multiplicative
solution~$j$. If the generators are all isometries
then~\cite[Theorem~3.12]{BeW14} allows us to complete the proof.
\end{proof}

\begin{example}
The non-commutative torus is determined by a choice of
$\lambda \in \T$. Given such~$\lambda$, we define unitary operators $U$ and $V$ on
$\ell^2( \Z^2 )$ by setting
\[
( U u )_{m, n} := u_{m + 1, n} \qquad \text{and} \qquad %
( V u )_{m, n} := \lambda^m u_{m, n + 1} \qquad %
\text{for all } u \in \ell^2( \Z^2 ) \text{ and } m, n \in \Z,
\]
and let $\alg_0$ be the $*$-algebra generated by $U$ and $V$. Note
that $U V = \lambda V U$ and $\alg$, the norm closure of $\alg_0$ in
$\bopp\bigl( \ell^2( \Z^2 ) \bigr)$, is then a concrete realisation of
the non-commutative torus with parameter $\lambda$, which is a
universal $C^*$~algebra. Lemma~\ref{lem:derivationsource} gives
existence of the derivations
\[
d_1 : \alg_0 \to \alg_0; \ U^m V^n \mapsto m U^m V^n %
\qquad \text{and} \qquad %
d_2 : \alg_0 \to \alg_0; \ U^m V^n \mapsto n U^m V^n,
\]
and the free flow generated by $\phi$ defined in
Theorem~\ref{thm:genfrombalance} is the one discussed
in~\cite[Theorem~6.9]{BeW14}.
\end{example}

\begin{example}
For each $N \in \N \cup \{ \infty \}$, the Cuntz algebra $\cuntz_N$ is
the $C^*$~algebra generated by isometries $\{ s_i \}_{i = 1}^N$ that
satisfy the additional relation $\sum_{i=1}^N s_i s_i^* = 1$. For
finite $N$, a concrete realisation can be given as follows: let
$\ini = \ell^2( \Z )$ with standard orthonormal basis
$\{ e_n \}_{n \in \Z}$ and define $s_i$ by continuous linear extension
of the map $s_i: e_n \mapsto e_{n N + i}$. A similar construction is
possible in the case when $N = \infty$. As has been known since its
introduction, all realisations of $\cuntz_N$ are mutually
isomorphic~\cite[Theorem~1.12]{Cun77} and so it is universal. Thus
Theorem~\ref{thm:genfrombalance} applies to give a flow on $\cuntz_N$,
where the dimension of the multiplicity space is $N$.
\end{example}

\begin{example}\label{ex:spheres}
More recent examples of universal $C^*$~algebras are the multitude of
non-commutative spheres studied in~\cite{Ban15}. The following are all
non-commutative examples for which all of the generators
$\{ z_i \}_{i = 1}^N$ are balanced:
\begin{align*}
C( S^{N - 1}_{\C, +} ) & = C^*\Bigl( \{ z_i \}_{i = 1}^N : %
\sum_{i = 1}^N z_i z_i^* = \sum_{i = 1}^N z_i^* z_i = 1 \Bigr); \\[1ex]
C( \overline{S}^{N - 1}_\C ) & = %
C( S^{N - 1}_{\C, +} ) / %
\langle \alpha \beta = -\beta \alpha \text{ for all distinct } a, b \in \{ z_i \}, \ %
\alpha \beta = \beta \alpha \text{ otherwise} \rangle; \\[1ex]
C( S^{N - 1}_{\C, **} ) & = %
C( S^{N - 1}_{\C, +} ) / \bigl\langle %
a b c = c b a \ \text{ for all } a, b, c \in \{ z_i, z_i^* \} \bigr\rangle; \\[1ex]
C( \overline{S}^{N - 1}_{\C, **} ) & = %
C( S^{N - 1}_{\C, +} ) / %
\langle \alpha \beta \gamma = -\gamma \beta \alpha \text{ for all distinct } a, b, c \in \{ z_i \}, \ %
\alpha \beta \gamma  = \gamma \beta \alpha \text{ otherwise} \rangle.
\end{align*}
The first of these is the \emph{complex free sphere}. If we also include the relations
$[ z_i, z_j ] = 0$ for all $i$ and $j$ then we have a commutative $C^*$~algebra
that is isomorphic to the algebra of continuous functions on the complex sphere,
but in the free sphere all commutativity has been eradicated.

The second algebra is the \emph{twisted} version of the sphere,
obtained by taking a quotient of the free sphere. The notation with
$a$ and $\alpha$ means that if $a = z_i$ then
$\alpha$ stands for either of $z_i$ or $z_i^*$. Thus we have the
imposed the following additional relations to those satisfied by the
free sphere: $z_i z_i^* = z_i^* z_i$ and, for $i \neq j$,
$z_i z_j = - z_j z_i$ and $z_i z_j^* = -z_j^* z_i$. The third and
fourth algebras are the \emph{half-liberations} of $S^{N - 1}_\C$ and
$\overline{S}^{N - 1}_\C$, respectively.

In~\cite{Ban15} there are also many real spheres studied, which have
generators $x_1$, \ldots, $x_N$ that are self adjoint and satisfy
$\sum_{i = 1}^N x_i^2 = 1$. The relation $x_i = x_i^*$
ensures that none of the generators are balanced, and so
Lemma~\ref{lem:derivationsource} is not applicable.
\end{example}

\begin{proposition}
For each of the four spheres from \tu{Example~\ref{ex:spheres}} the
gauge-free generator $\phi$ defined by~(\ref{eqn:delta+tau}) is well
defined and the solution $j$ to~(\ref{eqn:flowqsde}) for this
generator is $*$-homomorphic in the strong sense.
\end{proposition}
\begin{proof}
Following the proof of Theorem~\ref{thm:genfrombalance}, we get a
weakly multiplicative, strong solution by applying the results
from~\cite{BeW14}. However, since the generators $z_i$ are not
isometries, we cannot apply \cite[Theorem~3.12]{BeW14} directly. This
problem is overcome by noting that since $j$ is unital and weakly
multiplicative we have that
\[
0 \le %
\sum_{i = 1}^N \langle j_t( z_i ) \xi, j_t( z_i ) \xi \rangle = %
\langle \xi, %
j_t \biggl( \sum_{i = 1}^N z_i^* z_i \biggr) \xi \rangle = %
\| \xi \|^2 \qquad \textrm{for any } t \in \R_+ \text{ and } \xi \in \ini \algten \evecs.
\]
Consequently each map $j_t( z_i )$ is contractive, and we can consider for
fixed $t$ the $C^*$~subalgebra of $\bop{\ini \otimes \fock}{}$
generated by these bounded operators. By universality, this algebra is
isomorphic to the relevant sphere, and so $j_t$ is indeed a
well-defined $*$-homomorphism.
\end{proof}

\begin{remark}
Let $\alg$ be a $C^*$~algebra as in
\tu{Lemma~\ref{lem:derivationsource}}, with generators $a_1$, \ldots,
$a_N$ being balanced, and in addition possibly having other
generators. Let $\alpha^1$, \ldots, $\alpha^N$ be the $C_0$~groups
associated to the given balanced generators. Together these define an
action of $\R^N$ on $\alg$:
\[
\alpha_\bt := %
\alpha^1_{t_1} \comp \cdots \comp \alpha^N_{t_N} \quad %
\text{for } \bt = ( t_1, \ldots, t_N ) \in \R^N,
\]
since
$\alpha^i_s \comp \alpha^j_t = \alpha^j_t \comp \alpha^i_s$ for all
$i$, $j$, $s$ and $t$. Indeed, periodicity of the groups actually
implies that we have an action of the $N$-torus $\T^N$ on $\alg$.

Let $\bw = ( w^1, \ldots, w^N )$ be an $N$-dimensional Wiener process,
and pick $c_1$, \ldots, $c_N \in \C$. We can randomise the group
action to define a mapping process $l$ as follows:
\begin{equation}\label{eqn:randomaction}
l_t( x ) := \alpha_{\bw_t}( x ) = %
\bigl( \alpha^1_{| c_1 | w^1_t} \comp \cdots \comp \alpha^N_{| c_N | w^N_t} \bigr)( x ) %
\qquad \textrm{for any } x \in \alg.
\end{equation}
Furthermore, if $j$ is the flow on $\alg$ from
Theorem~\ref{thm:genfrombalance} then, writing
$c_j = \I |c_j| e^{\I \theta_j}$ for $\theta_j \in [ 0, 2 \pi )$, we
have that
\[
\rd j_t( x ) = j_t\Bigl( -\hlf \sum_{i = 1}^N | c_j |^2 d_j^2( x ) \Bigr) \std t + %
\sum_{j = 1}^N j_t\bigl( \I | c_j | d_j( x ) \bigr) %
\bigl( e^{\I \theta_j} \std( A^j )^*_t + %
e^{-\I \theta_j} \std ( A^j )_t \bigr)
\]
where $x \in \alg_0$ and $A^j$ is the $j$th component of the annihilation operator with
respect to the given basis $\{e_1, \ldots, e_N\}$ of $\mul$. There
is a natural identification of the Fock space $\fock$ with the $L^2$-space
for $\bw$ (see~\cite{Lin05}), and moreover each operator process
$B^j_t := e^{\I \theta_j} ( A^j )^*_t + e^{\I \theta_j} ( A^j )_t$ can
then be viewed as a realisation of the corresponding component of $\bw$. Applying
the usual It\^{o} Lemma to~(\ref{eqn:randomaction}), and using these
identifications, it follows that $l_t( x ) \cong j_t( x )$, so that
the flows arising from Theorem~\ref{thm:genfrombalance} have a
classical origin.
\end{remark}

\begin{example}\label{ex:gaugefreeperturbations}
Let $\phi$ be the gauge-free generator in standard form from
Theorem~\ref{thm:genfrombalance} and let~$j$
be the free flow generated by $\phi$. Suppose
$F = \left[ \begin{smallmatrix} k & -l^* \\[0.5ex] l & 0 \end{smallmatrix}
\right] \in \alg \matten \bop{\mmul}{}$ is a gauge-free multiplier generator,
where $k \in \alg$ and $l \in \alg \matten \bop{\mul}{}$. To satisfy the conditions
$q( F ) = q( F^* ) = 0$, we require that $k + k^* + l^* l = 0$. If
$X$ is the solution to~(\ref{eqn:multqsde}) and
$k_t = X^*_t j_t( \cdot ) X_t$ then $k$ is a weak solution to the
QSDE~(\ref{eqn:perturbedQSDE}) for the generator $\psi$, where
\[
\psi( x ) = \begin{bmatrix}
 \tau( x ) + l^* \delta( x ) + l^* ( x \otimes I_\mul ) l + %
\delta^\dagger( x ) l + k^* x + x k & %
\delta^\dagger( x ) + l^* ( x \otimes I_\mul )  -x l^* \\[1ex]
 \delta( x ) + ( x \otimes I_\mul ) l - l x
 & 0
\end{bmatrix}.
\]
If the set $\{ a_i: i \in I \}$ of generators of $\alg$ is countable
then $\mul$ is separable, and we may assume that $\ini$ is also
separable~\cite[Corollary~3.7.5]{Ped79}. The process $k$ is then a strong
solution to~(\ref{eqn:perturbedQSDE}).

Letting $l_j := \bigl( I_\ini \otimes \bra{e_j} \bigr) l \in \alg$, we have that
\[
\delta( x ) +( x \otimes I_\mul ) l - l x = %
\sum_{j \in J} \bigl( c_j d_j( x ) + d_{l_j}( x ) \bigr),
\]
where
\[
d_r( x ) := [ x, r ] = x r - r x \qquad \text{for any } r \in \alg.
\]
Feynman--Kac perturbation techniques similar to those developed here
were employed in~\cite{CGS03} as a
means of constructing possible Laplacians for the non-commutative
torus, in which the components $c_j d_j + d_{l_j}$ should be
thought of as Dirac operators. More precisely, in~\cite{CGS03} they
studied operators on the Hilbert space $L^2( \alg )$ arising from the
unique trace on the non-commutative torus. To fit into the framework
of non-commutative geometry, the operators on this space arising from
these derivations ought to be self adjoint, and so the component
derivations should be symmetric. This forces the choice of
$c_j = \I \beta_j$ for some $\beta_j \in \R$ when looking at
the generator of the free flow. For the perturbed generator, since
$(d_r)^\dagger = -d_{r^*}$, it is appropriate to make the choice $l_j = \I n_j$,
where $n_j \in \alg$ is self adjoint.

Under the assumption that $c_j = \I \beta_j$, where $\beta_j \in \R$, the unperturbed
semigroup generator
\[
\tau = -\frac{1}{2} \sum_{j \in J} ( \beta_j d_j )^2.
\]
If
\[
\tau'( x ) = %
\tau( x ) + l^* \delta( x ) + l^* ( x \otimes I_\mul ) l + %
\delta^\dagger( x ) l + k^* x + x k \qquad \textrm{and} \qquad
d_j' = \I ( \beta_j d_j + d_{n_j} )
\]
then we can ask if the perturbed semigroup generator satisfies the
analogous equation:
\begin{equation}\label{eqn:Lapsumsquares}
\tau' = -\hlf \sum_{j \in J} ( d_j' )^2 = %
-\hlf \sum_{j \in J} ( \beta_j d_j + d_{n_j} )^2.
\end{equation}
If $J$ is finite then equation~(\ref{eqn:Lapsumsquares}) holds
if each $n_j \in \alg_0$ is self adjoint and we choose
\begin{equation}\label{eqn:correctk}
k = -\hlf %
\sum_{j \in J} \bigl( \beta_j d_j( n_j ) +n_j^2 \bigr).
\end{equation}
The requirement that $q( F ) = 0$ ensures that the value of $k + k^*$ must be
$-l^* l$, and this is consistent with~(\ref{eqn:correctk}). However, we can also see
from~(\ref{eqn:correctk}) that $k - k^* = \I \sum_j \beta_j d_j( n_j )$,
so the imaginary part of $k$ is not completely arbitrary.
\end{example}


\section{References}

\begin{thebibliography}{99}

\bibitem{Acc78}
\textsc{L.~Accardi},
\textit{On the quantum Feynman-Kac formula},
Rend.\ Sem.\ Mat.\ Fis.\ Milano~48 (1978), 135--180.

\bibitem{AcK01}
\textsc{L.~Accardi} \& \textsc{S.~V.~Kozyrev},
\textit{On the structure of Markov flows},
Chaos Solitons Fractals~12 (2001), no.~14--15, 2639--2655.

\bibitem{Ban15}
\textsc{T.~Banica},
\textit{Liberations and twists of real and complex spheres},
J.\ Geom.\ Phys.~96 (2015), 1--25.

\bibitem{Blt10}
\textsc{A.~C.~R.~Belton},
\textit{Random-walk approximation to vacuum cocycles},
J.\ Lond.\ Math.\ Soc.~(2)~81 (2010), no.~2, 412--434.

\bibitem{BLS12}
\textsc{A.~C.~R.~Belton}, \textsc{J.~M.~Lindsay} \& \textsc{A.~G.~Skalski},
\textit{A vacuum-adapted approach to quantum Feynman-Kac formulae},
Commun.\ Stoch.\ Anal.~6 (2012), no.~1, 95--109.

\bibitem{BLS13}
\textsc{A.~C.~R.~Belton}, \textsc{J.~M.~Lindsay} \& \textsc{A.~G.~Skalski},
\textit{Quantum Feynman--Kac perturbations},
J.\ London Math.\ Soc.\ (2)~89 (2014), no.~1, 275--300.

\bibitem{BeW14}
\textsc{A.~C.~R.~Belton} \& \textsc{S.~J.~Wills},
\textit{An algebraic construction of quantum flows with unbounded generators},
Ann.\ Inst.\ Henri Poincar\'e Probab.\ Stat.~51 (2015), no.~1, 349--375.

\bibitem{CGS03}
\textsc{P.~S.~Chakraborty}, \textsc{D.~Goswami} \&
\textsc{K.~B.~Sinha},
\textit{Probability and geometry on some noncommutative manifolds},
J.\ Operator Theory~49 (2003), no.~1, 185--201.

\bibitem{Cun77}
\textsc{J.~Cuntz},
\textit{Simple $C^*$-algebras generated by isometries},
Comm.\ Math.\ Phys.~57 (1977), no.~2, 173--185.

\bibitem{DaS92} 
\textsc{P.~K.~Das}, \& \textsc{K.~B.~Sinha},
\textit{Quantum flows with infinite degrees of
freedom and their perturbations},
in: Quantum Probability and Related Topics VII,
L.~Accardi (ed.), World Scientific, Singapore 1992, 109--123.

\bibitem{Dav07}
\textsc{E.~B.~Davies},
\textit{Linear operators and their spectra},
Cambridge University Press, Cambridge, 2007.

\bibitem{EvH90}
\textsc{M.~P.~Evans} \& \textsc{R.~L.~Hudson},
\textit{Perturbations of quantum diffusions},
J.\ London Math.\ Soc.\ (2)~41 (1990), no.~2, 373--384.

\bibitem{GLSW03}
\textsc{D.~Goswami}, \textsc{J.~M.~Lindsay}, \textsc{K.~B.~Sinha} \&
\textsc{S.~J.~Wills},
\textit{Dilation of Markovian cocycles on a von Neumann algebra},
Pacific J.\ Math.~211 (2003), no.~2, 221--247.

\bibitem{GLW01}
\textsc{D.~Goswami}, \textsc{J.~M.~Lindsay} \& \textsc{S.~J.~Wills},
\textit{A stochastic Stinespring theorem},
Math.\ Ann.~319 (2001), no.~4, 647--673.

\bibitem{Lin05}
\textsc{J.~M.~Lindsay},
\textit{Quantum stochastic analysis --- an introduction},
in: Quantum Independent Increment Processes~I,
M.~Sch\"urmann \& U.~Franz (eds.), Lecture Notes in Mathematics~1865,
Springer, Berlin, 2005, 181--271.

\bibitem{LiP98}
\textsc{J.~M.~Lindsay} \& \textsc{K.~R.~Parthasarathy},
\textit{On the generators of quantum stochastic flows},
J.\ Funct.\ Anal.~158 (1998), no.~2, 521--549.

\bibitem{LiW00a}
\textsc{J.~M.~Lindsay} \& \textsc{S.~J.~Wills},
\textit{Existence, positivity and contractivity for quantum stochastic
flows with infinite dimensional noise},
Probab.\ Theory Related Fields~116 (2000), no.~4, 505--543.

\bibitem{LiW00b}
\textsc{J.~M.~Lindsay} \& \textsc{S.~J.~Wills},
\textit{Markovian cocycles on operator algebras adapted to a Fock
filtration},
J.\ Funct.\ Anal.~178 (2000), no.~2, 269--305.

\bibitem{LiW01}
\textsc{J.~M.~Lindsay} \& \textsc{S.~J.~Wills},
\textit{Existence of Feller cocycles on a $C^*$-algebra},
Bull.\ London Math.\ Soc.~33 (2001), no.~5, 613--621.

\bibitem{LiW14}
\textsc{J.~M.~Lindsay} \& \textsc{S.~J.~Wills},
\textit{Quantum stochastic cocycles and completely bounded semigroups
on operator spaces},
Int.\ Math.\ Res.\ Not.\ IMRN~2014 (2014), no.~11, 3096--3139.

\bibitem{LiW21}
\textsc{J.~M.~Lindsay} \& \textsc{S.~J.~Wills},
\textit{Quantum stochastic cocycles and completely bounded semigroups
on operator spaces~II},
Comm.\ Math.\ Phys.~383 (2021), no.~1, 153--199.

\bibitem{Ped79}
\textsc{G.~K.~Pedersen},
\textit{$C^*$-algebras and their automorphism groups},  Academic
Press, London--New York, 1979.

\bibitem{Reb05}
\textsc{R.~Rebolledo},
\textit{Decoherence of quantum Markov semigroups},
Ann.\ Inst.\ Henri Poincar\'e Probab.\ Stat.~41 (2005), no.~3,
349--373.

\end{thebibliography}
\end{document}